\documentclass{amsart}

\usepackage{graphicx}
\usepackage{psfrag}
\usepackage{latexsym}
\usepackage{amssymb}
\usepackage{amsfonts}
\usepackage{amsmath}
\usepackage{amsthm}

\newcommand\N{{\mathbb N}}

\newcommand\R{{\mathbb R}}

\newcommand\mH{{\mathcal H}}
\newcommand\Sph{{\mathbb S}}
\newcommand\dee{\partial}

\renewcommand\:{\colon}

\newtheorem{theorem}{Theorem}[section]
\newtheorem{corollary}[theorem]{Corollary}
\newtheorem{lemma}[theorem]{Lemma}
\newtheorem{proposition}[theorem]{Proposition}

\theoremstyle{definition}

\theoremstyle{remark}
\newtheorem{remark}[theorem]{Remark}
\newtheorem{question}[theorem]{Question}
\numberwithin{equation}{section}

\newcommand{\ep}{\epsilon}
\newcommand{\del}{\delta}

\newcommand{\reals}{\mathbb{R}}
\newcommand{\ints}{\mathbb{Z}}
\newcommand{\nats}{\mathbb{N}}
\newcommand{\comps}{\mathbb{C}}

\newcommand{\disk}{{\mathbb{D}^2}}

\newcommand{\sphere}{{\mathbb{S}^2}}
\newcommand{\Sircle}{{\mathbb{S}^1}}

\newcommand{\nbhd}{\mathcal{N}}

\newcommand{\inv}{^{-1}}

\newcommand{\ovl}{\overline}

\newcommand{\into}{\hookrightarrow}

\newcommand{\til}{\widetilde}

\newcommand{\wh}{\widehat}

\newcommand{\Hdim}{\mathcal{H}}
\newcommand{\subeq}{\subseteq}
\newcommand{\supeq}{\supseteq}
\newcommand{\bslash}{\backslash}
\newcommand{\mand}{\quad \hbox{and} \quad}

\newcommand{\eequiv}{\overset{e}{\sim}}
\newcommand{\mbf}{\mathbf}

\def\diam{\operatorname{diam}}
\def\card{\operatorname{card}}
\def\dist{\operatorname{dist}}
\def\im{\operatorname{im}}

\def\cl{\operatorname{cl}}

\def\XXint#1#2#3{{\setbox0=\hbox{$#1{#2#3}{\int}$} 
\vcenter{\hbox{$#2#3$}}\kern-.5\wd0}}

\begin{document}
\title{Quasisymmetric Koebe Uniformization}
\author{S. Merenkov}
\address{Department of Mathematics, University of Illinois at Urbana-Champaign, 1409 W Green Street, Urbana, IL 61801, USA}
\thanks{S.~M. was supported by NSF grants DMS-0653439 and DMS-1001144}
\email{merenkov@illinois.edu}
\author{K. Wildrick}
\address{Mathematisches Institut, Universit\"at Bern, Sidlerstrasse 5, 3012 Bern, Switzerland}
\thanks{K.~W. was supported by Academy of Finland grants 120972 and 128144}
\email{kevin.wildrick@math.unibe.ch}
\begin{abstract}
We study a quasisymmetric version of the classical Koebe uniformization theorem in the context of Ahlfors regular metric surfaces. In particular, we prove that an Ahlfors 2-regular metric surface $X$ homeomorphic to a finitely connected domain in the standard 2-sphere $\Sph^2$ is quasisymmetrically equivalent to a circle domain in $\Sph^2$ if and only if $X$ is linearly locally connected and its completion is compact. We also give a counterexample in the countably connected case.   
\end{abstract}

\date{}
\maketitle

\section{Introduction}

Uniformization problems are amongst the oldest and most important problems in mathematical analysis.  A premier example is the measurable Riemann mapping theorem, gives a robust existence theory for quasiconformal mappings in the plane. A quasiconformal mapping between domains in a Euclidean space is a homeomorphism that sends infinitesimal balls to infinitesimal ellipsoids of uniformly bounded ellipticity. The theory of quasiconformal mappings has been one of the most fruitful in analysis, yielding applications to hyperbolic geometry, geometric group theory, complex dynamics, partial differential equations, and mathematical physics. 

In the past few decades, many aspects of the theory of quasiconformal mappings have been extended to apply to abstract metric spaces.  A key factor in these developments has been the realization that in metric spaces with controlled geometry, the infinitesimal condition imposed by quasiconformal mappings actually implies a stronger global condition called quasisymmetry \cite{Acta}. The fact that quasisymmetric mappings are required to have good behavior at all scales makes them well suited to metric spaces that a priori have no useful infinitesimal structure.

A homeomorphism $f\colon X \to Y$ of metric spaces is \emph{quasisymmetric} if there is a homeomorphism $\eta\colon [0,\infty) \to [0,\infty)$ such that if $x, y, z$ are distinct points of $X$, then 
$$\frac{d_Y(f(x),f(y))}{d_Y(f(x),f(z))}\leq \eta \left(\frac{d_X(x,y)}{d_X(x,z)}\right).$$
The homeomorphism $\eta$ is called a distortion function of $f$.  If we wish to emphasize that a quasisymmetric mapping has a particular distortion function $\eta$, we will call it $\eta$-quasisymmetric.  
   
Despite the highly developed machinery for quasiconformal analysis on metric spaces, an existence theory for quasisymmetric mappings on metric spaces analogous to that of conformal mappings on Riemann surfaces has only recently been explored.  The motivation for such results arises from geometric group theory \cite{B-K}, the dynamics of rational maps on the sphere \cite{Expand}, and the analysis of bi-Lipschitz mappings and rectifiable sets in Euclidean space \cite{QCJ}.

More than a decade after foundational results of Tukia and V\"ais\"al\"a in dimension one \cite{QS}, Bonk and Kleiner \cite{B-K} gave simple sufficient conditions for a metric space to be quasisymmetrically equivalent to the standard 2-sphere $\sphere$. 

\begin{theorem}[Bonk--Kleiner] \label{two sphere} Let $X$ be an Ahlfors $2$-regular metric space homeomorphic to the sphere $\sphere$.  Then $X$ is quasisymmetrically equivalent to the sphere $\sphere$ if and only if $X$ is linearly locally connected.
\end{theorem}

The condition that $X$ is linearly locally connected (LLC), which heuristically means that $X$ does not have cusps, is a quasisymmetric invariant. Ahlfors $2$-regularity, which states that the two-dimensional Hausdorff measure of a ball is uniformly comparable to the square of its radius, is not. See Section \ref{notation} for precise definitions. A version of Theorem \ref{two sphere} for metric spaces homeomorphic to the plane was derived in \cite{QSPlanes}, and a local version given in \cite{QSStructures}. 

In this paper, we seek a version of Theorem \ref{two sphere} for domains in $\sphere$.  The motivation for our inquiry comes from the Kapovich--Kleiner conjecture of geometric group theory, described in Section \ref{geo group}, and from analogous classical conformal uniformization theorems onto circle domains.

A \emph{circle domain} is a domain $\Omega \subeq \sphere$ such that each component of $\sphere \bslash \Omega$ is either a round disk or a point. In 1909 \cite{Koebe}, Koebe posed the following conjecture, known as the Kreisnormierungsproblem: every domain in the plane is conformally equivalent to a circle domain. In the 1920's \cite{Koebe2}, Koebe was able to confirm his conjecture in the finitely connected case.

\begin{theorem}[Koebe's uniformization onto circle domains]\label{Koebe's thm}  Let $\Omega \subeq \comps$ be a domain with finitely many complementary components.  Then $\Omega$ is conformally equivalent to a circle domain.  
\end{theorem}

We first state a quasisymmetric version of Koebe's theorem, which we attain as a consequence of our main result. Denoting the completion of a metric space $X$ by $\ovl{X}$, we define the \emph{metric boundary} of $X$ by $\partial{X}=\ovl{X}\bslash X$.  We say that a component of $\partial{X}$ is non-trivial if it contains more than one point.

\begin{theorem}\label{finite case} Let $(X,d)$ be an Ahlfors $2$-regular metric space that is homeomorphic to a domain in $\sphere$, and such that $\partial{X}$ has finitely many non-trivial components. Then $(X,d)$ is quasisymmetrically equivalent to a circle domain if and only if $(X,d)$ is linearly locally connected and the completion $\ovl{X}$ is compact. 
\end{theorem}

Theorem \ref{finite case} is only quantitative in the sense that the distortion function of the quasisymmetric mapping may be chosen to depend only on the constants associated to the various conditions on $X$ and the ratio of the diameter of $X$ to the minimum distance between components of $\partial{X}$.

In 1993, He and Schramm confirmed Koebe's conjecture in the case of countably many complementary components \cite{He}. In full generality the conjecture remains open.  A key tool in He and Schramm's proof was transifinte induction on the rank of the boundary of a domain $\Omega$ in $\sphere$, which measures the complexity with which components of the boundary converge to one another.  The rank of a collection of boundary components is defined via a canonical topology on the set of components of the boundary; see Section \ref{comps and ends}.  It is shown there that if a metric space $(X,d)$ is homeomorphic to a domain $\Omega$ in $\sphere$ and is linearly locally connected, then the natural topology on the set of components of the metric boundary is homeomorphic to the natural topology on the set of boundary components of $\Omega$. This allows us to define rank as in the classical setting.  We denote the topologized collection of components of $\partial{X}$ by $\mathcal{C}(X)$.

In the following statement, which is our main result, we consider quasisymmetric uniformization onto circle domains $\Omega$ with the property that $\mathcal{C}(\Omega)$  is \emph{uniformly relatively separated}, meaning there is a uniform lower bound on the \emph{relative distance}
$$\bigtriangleup(E,F)= \frac{\dist(E,F)}{\min\{\diam(E),\diam(F)\}},$$
between any pair of non-trivial boundary components. This condition is appears naturally in both classical quasiconformal analysis and geometric group theory. Moreover, we employ \emph{annular} linear local connectedness (ALLC), which is more natural than the $\rm{LLC}$ condition in this setting. 

\begin{theorem}\label{finite rank} Let $(X,d)$ be a metric space, homeomorphic to a domain in $\sphere$, such that the closure of the collection of non-trivial components of $\partial{X}$ is countable and has finite rank. Moreover, suppose that 
\begin{enumerate}
\item\label{2 reg condition} $(X,d)$ is Ahlfors $2$-regular,
\item\label{planarity} setting, for each integer $k \geq 0$, 
$$n_k= \sup \card\{E \in \mathcal{C}(X): E\cap B(x,r) \neq \emptyset \ \text{and}\ 2^{-k}r <\diam{E}\leq 2^{-k+1}r\},$$
where the supremum is taken over all $x \in X$ and $0<r<2\diam X$, 
it holds that 
$$\sum_{k=0}^{\infty}n_k 2^{-2k} < \infty.$$  
\end{enumerate}
Then $(X,d)$ is quasisymmetrically equivalent to a circle domain such that $\mathcal{C}(\Omega)$ is uniformly relatively separated if and only if $X$ has the following properties:
\begin{enumerate}\setcounter{enumi}{2}
\item\label{compact condition}  the completion $\ovl{X}$ is compact,
\item\label{ALLC condition} $(X,d)$ is annularly linearly locally connected,
\item\label{rel sep condition} $\mathcal{C}(X)$ is uniformly relatively separated.
\end{enumerate}
\end{theorem}

Theorem \ref{finite rank} is quantitative in the sense that the distortion function of the quasisymmetric mapping may be chosen to depend only on the constants associated to the various conditions on $X$, and vice-versa.

A key tool in our proof is the following similar uniformization result of Bonk, which is valid for subsets of $\sphere$ \cite{CarpetUnif}.
\begin{theorem}[Bonk]\label{two sphere Koebe} Let $\{S_i\}_{i \in \nats}$ be a collection of uniformly relatively separated uniform quasicircles in $\sphere$ that bound disjoint Jordan domains.  Then there is a quasisymmetric homeomorphism $f \colon \sphere \to \sphere$ such that for each $i \in \nats$, the set $f(S_i)$ is a round circle in $\sphere$. 
\end{theorem}

This result, which is also quantitative, allows us to conclude that if a metric space $(X,d)$ as in the statement of Theorem \ref{finite rank} satisfies conditions \eqref{2 reg condition}-\eqref{rel sep condition}, then it is quasisymmetrically equivalent to a circle domain with uniformly relatively separated boundary circles as soon as there is any quasisymmetric embedding of $X$ into $\sphere$. Thus, producing such an embedding is the main focus of this paper. 

It is of great interest to know if conditions \eqref{2 reg condition} and \eqref{planarity} can be replaced with conditions that are quasisymmetrically invariant. By snowflaking, i.e., raising the metric to power $0 < \alpha < 1$, the sphere $\sphere=\R^2\cup\{\infty\}$  in one direction only, say, in the direction of $x$-axis, one produces a metric space homeomorphic to $\sphere$ that satisfies all the assumptions of Theorem \ref{finite rank} (and Theorem~\ref{two sphere}) except for Ahlfors $2$-regularity, but fails to be quasisymmetrically equivalent to $\sphere$.  On the other hand, not every quasisymmetric image of $\sphere$ is Ahlfors $2$-regular, as is seen by the usual snowflaking of the standard metric on $\sphere$.

Our second main result is the existence of a metric space satisfying all assumptions of Theorem \ref{finite rank}, except for condition \eqref{planarity}, that fails to quasisymmetrically embed in $\sphere$.

\begin{theorem}\label{example} There is a metric space $(X,d)$, homeomorphic to a domain in $\sphere$, with the following properties
\begin{itemize}
\item $\partial{X}$ has rank $1$,
\item $X$ is Ahlfors $2$-regular,
\item the completion $\ovl{X}$ is compact,
\item $X$ is annularly linearly locally connected,
\item the components of $\partial{X}$ are uniformly relatively separated,
\item there is no quasisymmetric embedding of $X$ into $\sphere$.
\end{itemize}
\end{theorem}

We now outline the proof of Theorem \ref{finite rank} and the structure of the paper. In Section \ref{comps and ends} we establish a topological characterization of the boundary components of a metric space as the \emph{space of ends} of the underlying topological space, at least in the presense of some control on the geometry of the space.  This allows us to develop a notion of rank, and in Section \ref{crosscuts}, a theory of cross-cuts analogous to the classical theory.  A key tool in this development is the following purely topological statement: every domain in $\sphere$ is homeomorphic to a domain in $\sphere$ with totally disconnected complement. This folklore theorem is proven in Section \ref{TotDisCon}. In Section~\ref{boundary uniformization section}, we use cross-cuts and a classical topological recognition theorem for $\Sircle$ to uniformize the boundary components of $X$.  The resulting theorem, which generalizes \cite[Theorem 1.3]{QSPlanes}, may be of independent interest:

\begin{theorem}\label{boundary uniformization} Suppose that $X$ is a metric space that is homeomorphic to a domain in $\sphere$, has compact completion, and satisfies the $\lambda$-$\rm{LLC}$ condition for some $\lambda \geq 1$.  Then each non-trivial component of $\partial{X}$ is a topological circle satisfying the $\lambda'$-$\rm{LLC}$ condition for some $\lambda' \geq 1$ depending only on $\lambda$. In particular, if the space $X$ is additionally assumed to be doubling, then each non-trivial component of $\partial{X}$ is quasisymmetrically equivalent to $\Sircle$ with distortion function depending only on $\lambda$ and the doubling constant. 
\end{theorem}

Suppose that $(X,d)$ is an Ahlfors $2$-regular and $\rm{ALLC}$ metric space that is homeomorphic to a domain in $\sphere$. Section \ref{completion uniformization section} shows that the completion $\ovl{X}$ is homeomorphic to the closure of an appropriately chosen circle domain. In Section \ref{porosity section} we prove general theorems implying that each non-trivial component of $\partial{X}$ has Assouad dimension strictly less than $2$, and hence, up to a bi-Lipschitz mapping, is the boundary of a planar quasidisk. We describe a general gluing procedure in Section \ref{gluing section}, and use it to ``fill in" the non-trivial components of $\partial{X}$. The resulting space $\wh{X}$ is $\rm{ALLC}$, but not always Ahlfors $2$-regular.  In order to guarantee this, we impose condition \eqref{planarity} of Theorem~\ref{finite rank}.  The assumption of finite rank allows us to reduce to the case that there are only finitely many components of $\partial{X}$. We show that $\wh{X}$ is homeomorphic to $\sphere$, and apply Theorem \ref{two sphere}.  The problem is now planar, and the above mentioned theorem of Bonk now yields the desired result. Section \ref{putting it together} summarizes our work and provides a formal proof of Theorems \ref{finite case} and \ref{finite rank}. 

Section \ref{counter-example section} is dedicated to proving Theorem \ref{example}. We conclude in Section \ref{problems} with discussion of related open problems.

\subsection{Acknowledgements}

We wish to thank Mario Bonk, Peter Feller, Pekka Koskela, John Mac\-kay, Daniel Meyer, Raanan Schul, and Jeremy Tyson for useful conversations and critical comments.  Some of the research leading to this work took place while the second author was visiting the University of Illinois at Urbana-Champaign and while both authors were visiting the State University of New York at Stony Brook. We are very thankful for the hospitality of those institutions. Also, the first author wishes to thank the Hausdorff Research Institute for Mathematics in Bonn, Germany, for its hospitality during the Rigidity program in the Fall 2009.

\section{Relationship to Gromov hyperbolic groups}\label{geo group}

In geometric group theory, there is a natural concept of a hyperbolic group due to Gromov.  These abstract objects share many features of Kleinian groups, including a boundary at infinity.  The boundary $\partial_{\infty}G$ of a Gromov hyperbolic group is equipped with a natural family of quasisymmetrically equivalent metrics. The structure of this boundary is categorically linked to the structure of the group: a large-scale bi-Lipschitz mapping between Gromov hyperbolic groups induces a quasisymmetric mapping between the corresponding boundaries at infinity, and vice versa.

One of the premier problems in geometric group theory is Cannon's conjecture, which states that for every Gromov hyperbolic group $G$ with boundary at infinity $\partial_{\infty}G$ homeomorphic to $\sphere$, there exists a discrete, co-compact, and isometric action of $G$ on hyperbolic 3-space.  In other words, if a Gromov hyperbolic group has the correct boundary at infinity, then it arises naturally from the corresponding situation in hyperbolic geometry.  By a theorem of Sullivan \cite{Sullivan}, Cannon's conjecture is equivalent to the following statement: if $G$ is a Gromov hyperbolic group, then $\partial_{\infty}G$ is homeomorphic to $\sphere$ if and only if $\partial_{\infty}G$ is quasisymmetrically equivalent to $\sphere$. If $\partial_\infty G$ is homeomorphic to $\sphere$, then each natural metric on $\partial_\infty G$ is $\rm{LLC}$ and Ahlfors $Q$-regular for some $Q \geq 2$.  If $Q=2$, then Theorem \ref{two sphere} confirms Cannon's conjecture. Indeed, even stronger statements are known \cite{ConfDim}.

Closely related to Cannon's conjecture is the Kapovich--Kleiner conjecture, which states that for every Gromov hyperbolic group $G$ with boundary at infinity $\partial_{\infty}G$ homeomorphic to the Sierpi\'nski carpet, there exists a discrete, co-compact, and isometric action of $G$ on a convex subset of hyperbolic 3-space with totally geodesic boundary.  The Kapovich--Kleiner conjecture is implied by Cannon's conjecture, and is equivalent to the following statement: if $G$ is a Gromov hyperbolic group, then $\partial_{\infty}G$ is homeomorphic to the Sierpi\'nski carpet if and only if $\partial_{\infty}G$ is quasisymmetrically equivalent to a round Sierpi\'nski carpet, i.e., to a subset of $\sphere$ that is homeomorphic to the Sierpi\'nski carpent and has peripheral curves that are round circles. 

Theorem \ref{finite rank} can be seen as a uniformization result for domains that might approximate a Sierpi\'nski carpet arising as the boundary of a hyperbolic group. If $\partial_\infty G$ is homeomorphic to the Sierpi\'nski carpet, then it is $\rm{ALLC}$ and the peripheral circles are uniformly separated uniform quasicirlces. Recent work of Bonk  established the Kapovich--Kleiner conjecture in the case that $\partial_\infty G$ can be quasisymmetrically embedded in $\mathbb{S}^2$; see \cite{CarpetUnif}. As noted by Bonk--Kleiner in \cite{MarioICM}, this is true when the Assouad dimension of $\partial_\infty G$ is strictly less than two, a hypothesis analgous to condition \eqref{planarity} in Theorem \ref{finite rank}. This observation and its proof provided ideas that will be used in Section \ref{gluing section}.

\section{Notation and basic results}\label{notation}
\subsection{Metric Spaces}
We are often concerned with conditions on a mapping or space that involve constants or distortion functions.  These constants or distortion functions are refered to as the \emph{data} of the conditions.  A theorem is said to be \emph{quantitative} if the data of the conclusions of theorem depend only on the data of the hypotheses.  In the proof of quantitative theorems, given non-negative quantities $A$ and $B$, we will employ the notation $A \lesssim B$ if there is a quantity $C \geq 1$, depending only on the data of the conditions in the hypotheses, such that $A \leq CB$.  We write $A \simeq B$ if $A \lesssim B$ and $B \lesssim A$.  

We will often denote a metric space $(X,d)$ by $X$.  Given a point $x \in X$ and a number $r>0$, we define the open and closed balls centered at $x$ of radius $r$ by
$$B_{(X,d)}(x, r) = \{y \in X: d(x,y) < r\} \mand \ovl{B}_{(X,d)}(x,r)= \{y \in X: d(x,y) \leq r\}.$$
For $0 \leq r < R$, we denote the open annulus centered at $x$ of inner radius $r$ and outer radius $R$ by
$$A_{(X,d)}(x,r,R) = \{y \in X: r < d(x,y) < R\}.$$
Note that when $r=0$, this corresponds to $B_{(X,d)}(x,R) \bslash \{x\}.$

Where it will not cause confusion, we denote $B_{(X,d)}(x,r)$ by $B_{X}(x,r)$, $B_d(x,r)$, or $B(x,r)$.  A similar convention is used for all other notions which depend implicitly on the underlying metric space. 

We denote the completion of a metric space $X$ by $\ovl{X}$, and define the metric boundary of $X$ by $\partial{X} =\ovl{X}\bslash X$.  These notions are not to be confused with their topological counterparts. 

For $\ep > 0$, the \emph{$\ep$-neighborhood} of a subset $E \subeq X$ is given by
$$\nbhd_{\ep}(E) = \bigcup_{x \in E} B(x, \ep).$$
The \emph{diameter} of $E$ is denoted by $\diam(E),$ and the distance between two subsets $E, F \subeq X$ is denoted by $\dist(E,F).$ If at least one of $E$ and $F$ has finite diameter, then the \emph{relative distance} of $E$ and $F$ is defined by 
$$\bigtriangleup(E,F)= \frac{\dist(E,F)}{\min\{\diam E, \diam F\}},$$
with the convention that $\bigtriangleup(E,F)=\infty$ if at least one of $E$ and $F$ has diameter $0$.

\begin{remark}\label{rel sep preservation} If $f \colon X \to Y$ is a quasisymmetric homeomorphism of metric spaces, and $E$ and $F$ are subsets of $X$, then 
$$\bigtriangleup(f(E),f(F)) \simeq \bigtriangleup(E,F).$$
This is easily seen using \cite[Proposition 10.10]{LAMS}. \end{remark}
 
Let $\sphere=\{(x,y,z)\in\R^3\colon x^2+y^2+z^2=1\}$ be the standard 2-sphere equipped with the restriction of the Euclidean metric on $\R^3$.
We say that $\Omega$ is a \emph{domain in $\sphere$} if it is an open and connected subset of $\sphere$. We always consider a domain in $\sphere$ as already metrized, i.e., equipped with the restriction of the standard spherical metric.  

\subsection{Dimension and measures} 

A metric space $X$ is \emph{doubling} if there is a constant $N \in \nats$ such that for any $x \in X$ and $r>0$, the ball $B(x,r)$ can be covered by at most $N$ balls of radius $r/2$.  This condition is quantitatively equivalent to the existence of constants $\alpha\geq 0$ and $C \geq 1$ such that $X$ is \emph{$(\alpha,C)$-homogeneous}, meaning that for every $x \in X$, and $0<r\leq R$, the ball $B(x,R)$ can be covered by at most $C(R/r)^{\alpha}$ balls of radius $r$.  The infinimum over all $\alpha$ such that $X$ is $(\alpha,C)$-homogeneous for some $C \geq 1$ is called the \emph{Assouad dimension} of $X$.  Hence, doubling metric spaces are precisely those metric spaces with finite Assouad dimension.   

%

In a doubling metric space, some balls may have lower Assouad dimension than the entire space.  To rule out this kind of non-homogeneity, one often employs a much stricter notion of finite-dimensionality.  The metric space $(X,d)$ is \emph{Ahlfors $Q$-regular}, $Q \geq 0$, if there is a constant $K \geq 1$ such that for all $x \in X$ and $0<r<\diam{X}$,
\begin{equation}\label{Ahlfors reg def}\frac{r^Q}{K} \leq \Hdim^Q_X(\ovl{B}(x,r)) \leq Kr^Q,\end{equation}
where $\Hdim^Q_X$ denotes the $Q$-dimensional Hausdorff measure on $X$. It is quantitatively equivalent to instead require that \eqref{Ahlfors reg def} hold for all open balls of radius less that $2\diam{X}$.  The existence of any Borel regular outer measure on $X$ that satisfies \eqref{Ahlfors reg def} quanitatively implies that $X$ is Ahlfors $Q$-regular; see \cite[Exercise 8.11]{LAMS}.

\begin{remark}\label{extend reg to boundary} Suppose that $(X,d)$ is Ahlfors $Q$-regular, $Q \geq 0$. Given $S \subeq \partial{X}$, the space $(X\cup S,d)$ is again Ahlfors $Q$-regular, quantitatively.  This is proven as in \cite[Lemma 2.11]{QSPlanes}. 
\end{remark}
\subsection{Connectivity conditions}

Here we describe various conditions that control the existence of ``cusps" in a metric space by means of connectivity. The basic concept of such conditions arose from the theory of quasiconformal mappings in the plane, where they play an important role as invariants. 

Let $\lambda \geq 1$.  A metric space $(X,d)$ is $\lambda$-linearly locally connected ($\lambda$-$\rm{LLC}$) if for all $a \in X$ and  $r >0$, the following two conditions are satisfied:
\begin{itemize}
\item[(i)]  for each pair of distinct points $x,y \in B(a,r)$, there is a continuum $E \subeq B(a,\lambda r)$ such that $x,y\in E$,  
\item[(ii)] for each pair of distinct points $x,y \in X\bslash B(a,r)$, there is a continuum $E \subeq X\bslash B(a, r/\lambda)$ such that $x,y \in E$. 
\end{itemize}
Individually, conditions $(i)$ and $(ii)$ are referred to as the $\lambda$-$\rm{LLC}_1$ and $\lambda$-$\rm{LLC}_2$ conditions, respectively. 

The $\rm{LLC}$ condition extends in a particularly nice way to the completion of a metric space.  We say that a metric space $(X,d)$ is $\lambda$-$\til{\rm{LLC}}$ if for all $a \in \ovl{X}$ and $r > 0$ the following conditions are satisfied:  
\begin{itemize} 
\item[(i)] For each pair of distinct points $x,y \in B_{\ovl{X}}(a,r)$, there is an embedding $\gamma\colon [0,1] \to \ovl{X}$ such that
$\gamma(0)=x$, $\gamma(1)=y$, $\gamma|_{(0,1)} \subeq X$, and $\im\gamma \subeq B_{\ovl{X}}(a,\lambda r),$  
\item[(ii)] For each pair of distinct points  $x,y \in \ovl{X}\bslash B_{\ovl{X}}(a,r)$, there is an embedding $\gamma\colon [0,1] \to \ovl{X}$ such that
$\gamma(0)=x$, $\gamma(1)=y$, $\gamma|_{(0,1)} \subeq X$, and $\im\gamma \subeq \ovl{X} \bslash B_{\ovl{X}}(a,r/\lambda).$
\end{itemize}
Individually, conditions $(i)$ and $(ii)$ are referred to as the $\lambda$-$\til{\rm{LLC}}_1$ and $\lambda$-$\til{\rm{LLC}}_2$ conditions, respectively. 

If a metric space $X$ is $\lambda$-$\til{\rm{LLC}}$, then it is also $\lambda$-$\rm{LLC}$.  The next proposition states that the two conditions are quantitatively equivalent for the spaces in consideration in this paper.  

\begin{proposition}\label{Better LLC}  Let $i \in \{1,2\}$, and let $(X,d)$ be a locally compact and locally path-connected metric space that satisfies the $\lambda$-$LLC_i$ condition.   Then $X$ is $\lambda'$-$\til{LLC}_i$, where $\lambda'$ depends only on $\lambda$.    
\end{proposition}
 
\begin{proof}  The key ingredient is the following statement:  If $U\subeq X$ is an open subset of $X$, and $E \subeq U$ is a continuum, then any pair of points $x,y \in E$ are contained in an arc in  $U$.  The details are straightforward and left to the reader. 
\end{proof}

Let $\lambda \geq 1$. A metric space $(X,d_X)$ is \emph{$\lambda$-annularly linearly locally connected} ($\lambda$-$\rm{ALLC}$) if for all points $a \in X$ and all $0\leq r<R$, each pair of distinct points in the annulus $A(a,r,R)$ is contained in a continuum in the annulus $A(a,r/\lambda, \lambda R).$ 

The $\rm{ALLC}$ condition forbids local cut-points in addition to ruling out cusps. For example, the standard circle $\mathbb{S}^1$ is $\rm{LLC}$ but not $\rm{ALLC}$. In our setting, the $\rm{ALLC}$ condition is a more natural assumption, and is in some cases equivalent to the $\rm{LLC}$ condition. 

We omit the proofs of the following three statements. The first is based on decomposing an arbitrary annulus into dyadic annuli. The second uses the fact that in a connected space, any distinct pair of points is contained in annulus around some third point. The third states that the $\rm{ALLC}$ condition extends to the boundary as in Proposition \ref{BetterALLC}. 

\begin{lemma}\label{ALLC r 2r} Let $\lambda \geq 1$. Suppose that a connected metric space $(X,d)$ satisfies the condition that for all points $a \in X$ and all $r>0$, each pair of disctinct points in the annulus $A(a,r,2r)$ is contained in a continuum in the annulus $A(a,r/\lambda, 2\lambda r).$ Then $X$ satisfies the $\lambda$-$\rm{ALLC}$ condition.
\end{lemma}

\begin{lemma}\label{ALLC gives LLC} Suppose that $(X,d)$ is a connected metric space that satisfies the $\lambda$-$\rm{ALLC}$ condition. Then $(X,d)$ satisfies the $2\lambda$-$\rm{LLC}$ condition.
\end{lemma}

\begin{proposition}\label{BetterALLC} Suppose that $(X,d)$ is a locally compact and locally path-connected metric space that satisfies the $\lambda$-$\rm{ALLC}$ condition. Then there is a quantity $\lambda' \geq 1$, depending only on $\lambda$, such that for all $a \in \ovl{X}$ and $0\leq r<R$, for each pair of distinct points $x,y \in A_{\ovl{X}}(a,r,R)$, there is an embedding $\gamma\colon [0,1] \to \ovl{X}$ such that
$\gamma(0)=x$, $\gamma(1)=y$, $\gamma|_{(0,1)} \subeq X$, and 
$$\im\gamma \subeq A_{\ovl{X}}(a,r/\lambda',\lambda'R).$$  
\end{proposition}

There is a close connection between the $\rm{ALLC}$ condition and the uniform relative separation of the components of the boundary of a given metric space. The following statement addresses only circle domains, but a more general result is probably valid. We postpone the proof until Section \ref{TotDisCon}.

\begin{proposition}\label{ALLC rel sep circ} Let $\Omega$ be a circle domain.  Then $\Omega$ satisfies the $\rm{ALLC}$ condition if and only if the components of $\partial{\Omega}$ are uniformly relatively separated, quantitatively.
\end{proposition}

In the case of metric spaces that are homeomorphic to domains in $\mathbb{S}^2$ with finitely many boundary components, the $\rm{LLC}$-condition may be upgraded to the $\rm{ALLC}$-condition. Again, we postpone the proof until Section \ref{TotDisCon}.

\begin{proposition}\label{LLC gives ALLC} Let $(X,d)$ be a metric space homeomorphic to a domain in $\sphere$, and assume that the boundary $\partial{X}$ has finitely many components. If $(X,d)$ is $\lambda$-$\rm{LLC}$, $\lambda \geq 1$, then it is $\Lambda$-$\rm{ALLC}$, where $\Lambda$ depends on $\lambda$ and the ratio of the diameter of $X$ to the minimum distance between components of $\partial{X}$.  
 \end{proposition}

\section{The space of boundary components of a metric space}\label{comps and ends}

In this section we assume that $(X,d)$ is a connected, locally compact metric space with the additional property that the completion $\ovl{X}$ is compact.  Note that as $X$ is locally compact, it is an open subset of $\ovl{X}$. Hence $\partial{X}$ is closed in $\ovl{X}$ and hence compact.

\subsection{Boundary components and ends}

Of course, the topological type of $\partial{X}$ depends on the specific metric $d$. However, the goal of this section is to show that under a simple geometric condition, the collection $\mathcal{C}(X)$ of \textit{components} of $\partial{X}$ depends only on the topological type of $X$.  

We define an equivalence relation $\sim$ on $\partial{X}$ by declaring that $x \sim y$ if and only if $x$ and $y$ are contained in the same component of $\partial{X}$.  Then there is a bijection between $\mathcal{C}(X)$ and the quotient $\partial{X}/\sim$, and hence we may endow $\mathcal{C}(X)$ with the quotient topology. Since $\partial{X}$ is compact, the space $\mathcal{C}(X)$ is compact as well.  Given a compact set $K \subeq X$ and a component $U$ of $X\bslash K$, denote
$$\mathcal{C}(K,U)=\{E \subeq \mathcal{C}(X): E \subeq (\partial{U} \cap \partial{X})\}.$$

Let $\mathcal{U}(X)$ denote the collection of sequences $\{x_i\} \subeq X$ with the property that for every compact set $K \subeq X$, there is a number $N \in \nats$ and a connected subset $U$ of $X\bslash K$ such that 
$$\{x_i\}_{i \geq N} \subeq U.$$
Define an equivalence relation $\eequiv$ on $\mathcal{U}(X)$ by $\{x_i\} \eequiv \{y_i\}$ if and only if the sequence $\{x_1,y_1,x_2,y_2,\hdots\}$ is in $\mathcal{U}(X)$.  An equivalence class $E$ defined by $\eequiv$ is called an \textit{end} of $X$, and we denote the collection of ends of $X$ by $\mathcal{E}(X)$.  

Given a compact subset $K \subeq X$ and a component $U$ of $X\bslash K$, define
$$\mathcal{E}(K,U) = \{[\{x_i\}] : \{x_i\} \in \mathcal{U}(X) \ \text{and}\ \exists \ N \in \nats \ \text{such that}\ \{x_i\}_{i\geq N} \subeq U\},$$
That this set is well-defined follows from the definition of the equivalence relation on $\mathcal{U}(X)$.  Let $\mathcal{B}$ be the collection of all such sets. 

\begin{proposition}\label{topology on ends} The collection $\mathcal{B}$ generates a unique topology on $\mathcal{E}(X)$ such that every open set is a union of sets in $\mathcal{B}$.
\end{proposition}

\begin{proof} We employ the standard criteria for proving generation \cite[Section 13]{Munkres}. As $X$ is connected, taking $K=\emptyset$ shows that $\mathcal{B}$ contains $\mathcal{E}(X)$.  Thus it suffices to show that given compact subsets $K_1$ and $K_2$ of $X$ and components $U_1$ and $U_2$ of $X\bslash K_1$ and $X\bslash K_2$ respectively, and given an end 
$$E \in \mathcal{E}(K_1,U_1) \cap \mathcal{E}(K_2,U_2),$$
there is a compact set $K$ and a component $U$ of $X \bslash K$ such that 
\begin{equation}\label{end topology crit} E \in \mathcal{E}(K,U) \subeq \left(\mathcal{E}(K_1,U_1) \cap \mathcal{E}(K_2,U_2)\right). \end{equation}

Let $\{x_i\} \in \mathcal{U}(X)$ represent the end $E$.  By definition of $\mathcal{U}(X)$, there is a component $U$ of $X\bslash K$, where $K=K_1 \cup K_2$, such that $\{x_i\}_{i \geq N} \subeq U$ for some $N \in \nats$. This implies that $E \in \mathcal{E}(K,U)$.  By assumption, there is a number $M \in \nats$ such that $\{x_i\}_{i \geq M}$ is contained in $U_1 \cap U_2$.  Since $U$ is a connected subset of $X \bslash K_1$ and $X \bslash K_2$, it follows that $U \subeq U_1 \cap U_2$.  This implies \eqref{end topology crit}. 
\end{proof}

We set the topology on $\mathcal{E}(X)$ to be that given by Proposition \ref{topology on ends}. It follows quickly from the definitions that $\mathcal{E}(X)$ is a Hausdorff space.

\begin{remark}\label{ends are topological} As ends are defined purely in topological terms, a homeomorphism $h \colon X \to Y$ to some other topological space $Y$ induces a homeomorphism $\phi \colon \mathcal{E}(X) \to \mathcal{E}(Y)$.  This homeomorphism is natural in the sense that a sequence $\{x_i\} \in \mathcal{U}(X)$ represents the end $E \in \mathcal{E}(X)$ if and only $\{h(x_i)\} \in \mathcal{U}(Y)$ represents the end $\phi(E) \in \mathcal{E}(Y)$. 
\end{remark}

To relate the ends of a metric space to the components of its metric boundary, we need an elementary result regarding connectivity. Let $\ep>0$ and let $x$ and $y$ be points in any metric space $Z$.  An \textit{$\ep$-chain} connecting $x$ to $y$ in $Z$ is a sequence $x=z_0,z_1,\hdots,z_n=y$ of points in $Z$ such that $d(z_j,z_{j+1})\leq \ep$ for each $j=0,\hdots,n-1$.  In compact spaces, the existence of arbitrarily fine chains can detect connectedness. We leave the proof of the following statement to that effect to the reader. 

\begin{lemma}\label{chain connect}  Let $x$ and $y$ be points in a metric space $Z$. If $x$ and $y$ lie in the same component of $Z$, then for every $\ep>0$ there is an $\ep$-chain in $Z$ connecting $x$ to $y$.  The converse statement holds if $Z$ is compact. 
\end{lemma}

%

An end always defines a unique boundary component. 

\begin{proposition}\label{end gives component}  If $\{x_i\}$ and $\{y_i\}$ are Cauchy sequences representing the same end of $X$, then they represent points in the same component of $\partial{X}$. 
\end{proposition}

\begin{proof}  Denote by $x$ and $y$ the points in $\partial{X}$ defined by $\{x_i\}$ and $\{y_i\}$, respectively.  By Lemma \ref{chain connect}, in order to show that $x$ and $y$ are contained in a single component of $\partial{X}$, it suffices to find an $\ep$-chain in $\partial{X}$ connecting $x$ to $y$ for every $\ep>0$. To this end, fix $\ep>0$.  Let 
$K$ be the compact subset of $X$ defined by  
\begin{equation}\label{ends components 1} K = X \bslash\nbhd_{\ovl{X}}\left(\partial{X}, \ep/3\right).\end{equation}  
By assumption, we may find $N \in \nats$ so large that $x_N$ and $y_N$ lie in a connected subset of $X \bslash K$, and that $d(x,x_N)$ and $d(y,y_N)$ are both less than $\ep/3$.  By Lemma~\ref{chain connect}, we may find an $\ep/3$-chain $x_N=z_0,\hdots,z_n=y_N$ in $X\bslash K$.  By \eqref{ends components 1}, for each $j=1,\hdots,n-1$ we may find a point $z_j' \in \partial{X}$ such that $d(z_j,z_j')<\ep/3$.  The triangle inequality now implies that $x,z_1',\hdots, z_{n-1}',y$ is an $\ep$-chain in $\partial{X}$, as required.
\end{proof}

\begin{remark}\label{define Phi} Proposition \ref{end gives component} allows us to define a map $\Phi \colon \mathcal{E}(X) \to \mathcal{C}(X)$ as follows.  Let $E \in \mathcal{E}(X)$ and let $\{x_i\}_{i \in \nats} \in \mathcal{U}(X)$ be a sequence representing $E$. As $\ovl{X}$ is compact, we may find a Cauchy subsequence $\{x_{i_j}\}$ of $\{x_i\}$ that represents a point in some boundary component $E' \in \mathcal{C}(X)$.  Set $\Phi(E)=E'$.  This is well defined by Proposition \ref{end gives component}.  
\end{remark}

We now consider when a boundary component $E \in \mathcal{C}(X)$ corresponds to an end.  
For subspaces of $\sphere$, this is always the case. The key tool in the proof of this is the following purely topological fact, which is mentioned in the proof of \cite[Lemma~2.5]{B-K}.

\begin{proposition}\label{good exhaustion} Each domain $\Omega$ in $\sphere$ may be written as a union of open and connected subsets $\Omega_1 \subeq \Omega_2 \subeq \hdots$ of $\Omega$ such that for each $i \in \nats$, the closure of $\Omega_i$ is a compact subset of $\Omega$ and $\partial{\Omega_i}$ is a finite collection of pairwise disjoint Jordan curves.
\end{proposition}

\begin{proposition}\label{planar component gives end} Let $\Omega$ be a domain in $\sphere$. If $x$ and $y$ are points in the same component of $\partial{\Omega}$, then any Cauchy sequences representing $x$ and $y$  are in $\mathcal{U}(\Omega)$ and represent the same end of $\Omega$.  
\end{proposition}  

\begin{proof} The metric boundary $\partial{\Omega}$ coincides with the usual topological boundary of $\Omega$ in $\sphere$. Let $\{\Omega_i\}$ denote the exhaustion of $\Omega$ provided by Proposition \ref{good exhaustion}.  Since $x$ and $y$ are in the same component of $\partial{\Omega}$, for each $i \in \nats$ they belong to a single simply connected component $U_i$ of $\sphere \bslash \Omega_i$. Let $\{x_i\}$ and $\{y_i\}$ be Cauchy sequences representing $x$ and $y$ respectively. 

Suppose that $K$ is a compact subset of $\Omega$.  We may find $i_0 \in \nats$ so large that $\Omega_{i_0} \supeq K$.  Then $U_{i_0} \cap \Omega$ does not intersect $K$.  Moreover, there is a number $N \in \nats$ such that 
$$\left(\{x_i\}_{i \geq N} \cup \{y_i\}_{i \geq N}\right) \subeq U_{i_0} \cap \Omega.$$
It now suffices to show that $U_{i_0} \cap \Omega$ is connected. Let $a$ and $b$ be points in $U_{i_0} \cap \Omega$.  We may find $j \geq i_0$ so large that $a$ and $b$ are contained in $\Omega_j$.  The set $U_{i_0} \cap \Omega_j$ is a simply connected domain with finitely many disjoint closed topological disks removed from its interior, and is therefore path-connected.  Thus $a$ and $b$ may be connected by a path inside $U_{i_0} \cap \Omega_j$, and hence inside $U_{i_0} \cap \Omega$.  
\end{proof}

The proof given above does not even pass to metric spaces that are merely homeomorphic to a domain in $\sphere$, as it need not be the case that the completion of such a space embeds topologically in $\sphere$.  However, under an additional assumption controlling the geometry of $X$, we can give a different proof. 

\begin{proposition}\label{LLC1 component gives end} Suppose that $X$ satisfies the $\rm{LLC}_1$ condition.  If $x$ and $y$ are points in the same component of $\partial{X}$, then any Cauchy sequences representing $x$ and $y$ are in $\mathcal{U}(X)$ and represent the same end of $X$.  
\end{proposition}

\begin{lemma}\label{connected nbhd of bdry} Suppose that $X$ satisfies the $\lambda$-$\rm{LLC}_1$ condition for some $\lambda \geq 1$.  Let $E$ be a connected subset of $\partial{X}$ and let $\ep>0$.  Then $\nbhd_{\ovl{X}}(E,\ep) \cap X$ is contained in a connected subset of $\nbhd_{\ovl{X}}(E,3\lambda \ep) \cap X$.  
\end{lemma}

\begin{proof} It suffices to show that if $x$ and $y$ are points in $\nbhd_{\ovl{X}}(E,\ep) \cap X$, then there is a continuum containing $x$ to $y$ inside of $\nbhd_{\ovl{X}}(E,3\lambda \ep) \cap X$.  Let $x'$ and $y'$ be points in $E$ such that $d(x,x')<\ep$ and $d(y,y')<\ep$.  By Lemma \ref{chain connect}, there is an $\ep$-chain $x'=z'_0,\hdots,z'_n=y'$ in $E$.  For each $j=1,\hdots,n-1$, find a point $z_j \in X$ such that $d(z_j,z'_j)<\ep$.  The triangle inequality implies that for $j=0,\hdots, n-1$,
$$z_{j+1} \in B_{X}(z_j, 3\ep).$$
Repeatedly applying the $\lambda$-$\rm{LLC}_1$ condition and concatenating now yields the desired result.
\end{proof}

\begin{proof}[Proof of Proposition \ref{LLC1 component gives end}] Let $x$ and $y$ be points in a connected subset $E$ of $\partial{X}$, and let $\{x_i\}$ and $\{y_i\}$ be Cauchy sequences in $X$ corresponding to $x$ and $y$, respectively. Let $K$ be a compact subset of $X$.  As $\ovl{X}$ is compact, we may find $\ep>0$ such that $\dist(E,K)>3\lambda\ep$.  Let $N \in \nats$ be so large that 
$$\{x_i\}_{i \geq N} \cup \{y_i\}_{i \geq N} \subeq \nbhd_{\ovl{X}}(E,\ep) \cap X.$$
Lemma \ref{connected nbhd of bdry} now implies the desired results.
\end{proof}

\begin{remark}\label{some thing needed} Proposition \ref{LLC1 component gives end} is not true without some control on the geometry of $X$.  The following example was pointed out to us by Daniel Meyer.  Let $(r,\theta,z)$ denote cylindrical coordinates on $\reals^3$, and set
$$X = \{(r,\theta, z): (r-1)^2 + z^2 = 1\} \bslash \{(0,0,0)\}.$$
Equipped with the standard metric inherited from $\reals^3$, the space $X$ is homeomorphic to a punctured disk, and hence has two ends.  However, the metric boundary of $X$ consists only of the point $\{(0,0,0)\}$. 
\end{remark}

The following statement is the main result of this section. We note that in the case that $X$ is a domain in $\sphere$, the statement is mentioned in \cite{He}.

\begin{theorem}\label{ends components summary} Suppose that $X$ is either a domain in $\sphere$ or satisfies the $\rm{LLC}_1$ condition.  Then the map $\Phi \colon \mathcal{E}(X) \to \mathcal{C}(X)$ defined in Remark \ref{define Phi} is a homeomorphism  that is natural in the sense that a Cauchy sequence $\{x_i\} \in \mathcal{U}(X)$ represents the end $E \in \mathcal{E}(X)$ if and only if it represents a point on the boundary component $\phi(E) \in \mathcal{C}(X)$.  
\end{theorem}

In the proof of the following lemma, we consider only the case that $X$ satisfies the $\rm{LLC}_1$ condition. If $X$ is a domain in $\sphere$, a proof is easily constructed using Proposition \ref{good exhaustion}.

\begin{lemma}\label{point nbhd} Suppose that $X$ is a domain in $\sphere$ or satisfies the $\rm{LLC}_1$ condition. Let $K$ be a compact subset of $X$ and let $U$ be a component of $X\bslash K$. Then the following statements hold:
\begin{itemize}
\item[(i)] for each $x \in \ovl{U} \cap \partial X$, there is a number $\del>0$ such that $B_{\ovl{X}}(x,\del)\cap X \subeq U$,
\item[(ii)] if $E \in \mathcal{C}(X)$ intersects $\partial{U} \cap \partial{X}$, then $E \in \mathcal{C}(K,U),$
\item[(iii)] the set $\mathcal{C}(K,U)$ is open in $\mathcal{C}(X)$.
\end{itemize}
\end{lemma}

\begin{proof} We assume that $X$ satisfies the $\lambda$-$\rm{LLC}_1$ condition for some $\lambda \geq 1$. 

Suppose that statement (i) is not true. Then for all sufficiently small $\del>0$ we may find points $a,b \in B_{\ovl{X}}(x,\del) \cap X$ such that $a \in U$ and $b$ is in some other component of $X \bslash K$. Using the $\lambda$-$\rm{LLC}_1$ condition to connect $a$ to $b$ inside of $B_X(a,2\lambda \del)$ produces a point of $K$ in the ball $B_{\ovl{X}}(x, 3\lambda\del)$. Letting $\del$ tend to $0$ produces a contradiction with the assumption that $K$ is a compact subset of $X$.  

Statement (i) implies that collection 
$$\{E \cap \partial{V}\cap\partial{X} : V \ \text{is a component of $X\bslash K$}\}$$
consists of pairwise disjoint open subsets of $E$.  Hence the connectedness of $E$ proves statement (ii).

Now, recall that $\mathcal{C}(X)$ is endowed with the quotient topology. Hence, by statement (ii), in order to show that $\mathcal{C}(K,U)$ is open in $\mathcal{C}(X)$, it suffices to show that $\bigcup_{E \in \mathcal{C}(K,U)} E$ is open in $\partial{X}$. This follows from statements (i) and (ii). 
\end{proof}

\begin{proof}[Proof of Theorem \ref{ends components summary}]
Proposition \ref{planar component gives end} or \ref{LLC1 component gives end} shows that $\Phi$ is injective. Given $E' \in \mathcal{C}(X)$ and a Cauchy sequence $\{x_i\}_{i \in \nats}$ representing a point in $E'$, Proposition \ref{planar component gives end} or \ref{LLC1 component gives end} also state that $\{x_i\}_{i \in \nats}$ is in $\mathcal{U}(X)$ and hence represents an end $E \in \mathcal{E}(X)$.  By definition, this implies that $\Phi(E)=E'$, and so $\Phi$ is surjective.

We now check that the bijection $\Phi$ is a homeomorphism.  Since $\mathcal{C}(X)$ is compact and $\mathcal{E}(X)$ is Hausdorff, this is true if $\Phi\inv$ is continuous. Hence, by Lemma~\ref{point nbhd} (iii) and the defintion of the toplogy on $\mathcal{E}(X)$,  it suffices to show that for any compact set $K \subeq X$ and any component $U$ of $X\bslash K$, 
$$\Phi(\mathcal{E}(K,U))=\mathcal{C}(K,U).$$
Let $E $ be an end in $\mathcal{E}(K,U)$. By definition, the limit of any Cauchy sequence representing $E$ lies in $\partial{U} \cap \partial{X}$. Hence, $\Phi(E)$ intersects $\partial{U} \cap \partial{X}$, and so Lemma \ref{point nbhd} (ii) shows that $\Phi(E) \in \mathcal{C}(K,U)$.  Now, let $E' \in \mathcal{C}(K,U)$ and choose a Cauchy sequence $\{x_i\}_{i \in \nats}$ representing a point in $E'$.  Lemma \ref{point nbhd} (i) implies that there is $N \in \nats$ such that  $\{x_i\}_{i \geq N}$ is contained in $U$. Again, Proposition \ref{planar component gives end} or \ref{LLC1 component gives end} states that $\{x_i\}$ is in $\mathcal{U}(X)$ and hence represents an end $E$. Thus, by definition, $E \in \mathcal{E}(K,U)$ and $\Phi(E)=E'$.
\end{proof}

\subsection{Rank}

We breifly recall the notion of rank as discussed in \cite{He}. Let $\mathcal{T}$ be a countable, compact, and Hausdorff topological space.  Set $\mathcal{T}^0=\mathcal{T}$, and for each $n\geq 1$, set $\mathcal{T}^n$, to be the set of non-isolated points of $\mathcal{T}^{n-1}$, and endow $\mathcal{T}^n$ with the subspace topology. This process can be continued using transfinite induction to define $\mathcal{T}^\alpha$ for each ordinal $\alpha$, though we will not have need for this. For each ordinal $\alpha$, the space $\mathcal{T}^\alpha$ is again countable, compact, and Hausdorff. By the Baire category theorem, there is a unique ordinal $\alpha$ such that $\mathcal{T}^\alpha$ is finite and non-empty; this ordinal is defined to be the \emph{rank} of $\mathcal{T}$. 

Let $(X,d)$ be a metric space that is either a domain in $\sphere$ or satisfies the $\rm{LLC}_1$- condition.  By Theorem \ref{ends components summary}, the space of boundary components $\mathcal{C}(X)$ is homeomorphic to the space of ends $\mathcal{E}(X)$. As mentioned above, the former is clearly compact and the latter clearly Hausdorff, hence both are compact Hausdorff spaces.  Hence, if $\mathcal{S}$ is closed and countable subset of $\mathcal{C}(X)$, the rank of $\mathcal{S}$ is defined.

\section{Domains with totally disconnected complement}\label{TotDisCon}

The simplest possible structure of a boundary component is that it consists of a single point.  The aim of this section is to show that if we are only concerned with the topological properties, we may always assume this is the case.  

\begin{proposition}\label{point boundary} Every domain in $\sphere$ is homeomorphic to a domain in $\sphere$ that has totally disconnected complement. 
\end{proposition}

To prove Proposition \ref{point boundary}, we employ the theory of decomposition spaces \cite{Daverman}.  A \emph{decomposition} of a topological space $S$ is simply a partition of $S$.  The \emph{non-degenerate} elements of a decomposition are those elements of the partition that contain at least two points.  The \emph{decomposition space} $S/G$ associated to a decomposition $G$ of a topological space $S$ is the topological quotient of $S$ obtained by, for each $g \in G$, identifying the points of $g$.  A decomposition $G$ is an \emph{upper semi-continuous decomposition} if each element is compact, and given any $g \in G$ and any open set $U \subeq S$ containing $g$, there is another open set $V$ containing $g$ with the property that if $g' \in G$ intersects $V$, then $g' \subeq U$.  If $G$ is an upper semi-continuous decomposition of a separable metric space $S$, then $S/G$ is a separable and metrizable space. \cite[Proposition~I.2.2]{Daverman}.

We will use one powerful theorem from classical decomposition space theory.  It identifies decompositions of $\sphere$ that are homeomorphic to $\sphere$ itself \cite{Moore}.  

\begin{theorem}[Moore]\label{moore} Suppose that $G$ is an upper semi-continuous decomposition of $\sphere$ with the property that for each $g \in G$, both $g$ and $\sphere \bslash g$ are connected.  Then $S/G$ is homeomorphic to $\sphere$. 
\end{theorem}

We also employ a powerful theorem of classical complex analysis.  It states that any domain in $\sphere$ can be mapped conformally (and hence homeomorphically) to a \emph{slit domain}, i.e., to a domain in $\sphere$ that is either complete, or whose complementary components are points or compact horizontal line segments in $\reals^2$ \cite[V.2]{Goluzin}.

\begin{theorem}\label{slits} Any domain in $\sphere$ is conformally equivalent to a slit domain.  
\end{theorem}

\begin{lemma}\label{usc} If $\Omega \subeq \sphere$ is a slit domain, then the components of $\reals^2\bslash \Omega$ containing at least two points form the non-degenerate elements of an upper semi-continuous decomposition of $\sphere$.  
\end{lemma}

\begin{proof} Let $G$ denote the decomposition of $\sphere$ whose non-degenerate elements are those components of $\reals^2 \bslash \Omega$ that contain at least two points.  Let $g\in G$, and let $U \subeq \sphere$ be an open set containing $g$.
Without loss of generality we may assume that $g=[0,1] \times \{0\}$ and that $U$ is an open and bounded subset of $\reals^2$ containing $g$.  

Set 
$$l = \max\{x : (x,0) \in \reals^2\bslash U\ \text{and}\ x<0\} \ \text{and}\ r = \min\{x : (x,0) \in \reals^2\bslash U\ \text{and}\ x>1\}.$$
Since $\Omega \cap U$ is open and $g$ is a compact and connected subset of $U$, there are closed, non-degenerate intervals $L,R \subeq \reals$ such that 
$$(L \times \{0\})\subeq ((l,0) \times \{0\})\cap \Omega \cap U\ \text{and}\ (R \times \{0\}) \subeq ((1,r) \times \{0\})\cap \Omega \cap U.$$  
If $G$ is not upper semi-continuous, then for every $n \in \nats$, we may find some horizontal line segment $g_n \in G$ and points $(x_n,y_n), (x_n',y_n) \in \reals^2$ such that 
$$(x_n,y_n) \in g_n \cap \nbhd(g, 1/n) \ \text{and} (x_n',y_n) \in g_n \cap (\sphere\bslash U).$$
After passing to a subsequence, we may assume that $(x_n,y_n)$ tends to a point of $g$.  Moreover, passing to another subsequence if needed, we may assume that either 
$$\limsup_{n \to \infty} x_n' \leq l \ \text{or} \liminf_{n \to \infty} x_n' \geq r.$$
We consider the latter case; a similar argument applies in the former.  For sufficiently large $n$, the point $x_n'$ is greater than any point in $R$, while $x_n$ is less than any point in $R$.  By the connectedness of $g_n$, we conclude that there is a point $(z_n, y_n) \in g_n$ with $z_n \in R$.  After passing to yet another subsequence, we may assume that $(z_n, y_n)$ converges to a point in $R \times \{0\}$.  This is a contradiction as $\Omega$ is open. See Figure \ref{USC}.\end{proof}

\begin{figure}[h]
\begin{center}
\psfrag{l}{$(l,0)$}
\psfrag{L}{$L$}
\psfrag{0}{$g=[0,1]\times \{0\}$}
\psfrag{U}{$U$}
\psfrag{R}{$R$}
\psfrag{r}{$(r,0)$}
\psfrag{g}{$g_n$}
\includegraphics[width=.50\textwidth]{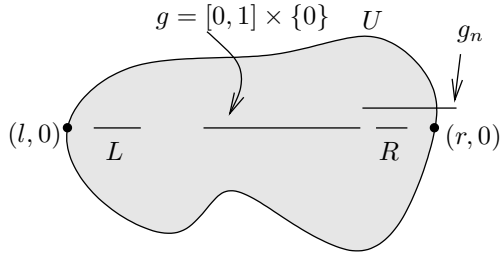}
\caption{If upper semi-continuity fails.}
\label{USC}
\end{center}
\end{figure}

\begin{proof}[Proof of Proposition \ref{point boundary}] Let $\Omega$ be a domain in $\sphere$.  By Theorem \ref{slits}, there is a homeomorpism $h_0 \colon \Omega \to \Omega_0$, where $\Omega_0$ is a slit domain.  By Lemma \ref{usc}, the components of $\sphere \bslash \Omega_0$ with at least two points form the non-degenerate elements of an upper semi-continuous decomposition $G$ of $\sphere$.  As each element of $G$ is either a point or a compact line segment in $\reals^2$, the hypotheses of Theorem \ref{moore} are satisfied.  Thus there is a homeomorphism $h_1 \colon \sphere/G \to \sphere.$  Let $\pi \colon \sphere \to \sphere/G$ denote the standard projection map.  By definition $\pi|_{\Omega_0}$ is a homeomorphism and $\pi(\sphere\bslash \Omega_0)$ is totally disconnected.  Thus $h_1 \circ \pi \circ h_0 \colon \Omega \to \sphere$ is a homeomorphism, and the image of $\Omega$ under this map has totally disconnected complement.  
\end{proof}

The ends of a domain in $\sphere$ with totally disconnected complement are particularly easy to understand: they are in bijection with the points of the complement, which are precisely the boundary components.

\begin{proposition}\label{point ends}  Suppose $\Omega$ is a domain in $\sphere$ with totally disconnected complement.  Then there is a homeomorphism $\phi \colon \mathcal{E}(\Omega) \to \partial{\Omega}$ with the property that a sequence $\{x_n\} \in \mathcal{U}(\Omega)$ represents the end $E \in \mathcal{E}(\Omega)$ if and only if it converges to $\phi(E)$.  \end{proposition}

\begin{proof} Since a totally disconnected subset of $\sphere$ cannot have interior, we see that $\sphere\bslash \Omega = \partial{\Omega}.$ Moreover, it is clear that $\mathcal{C}(\Omega)$ and $\partial{\Omega}$ are naturally homeomorphic. Hence Theorem \ref{ends components summary} provides the desired homeomorphism.
\end{proof}
 
The following statement transfers the work of this section to the general setting. For the remainder of this section, we assume that $(X,d)$ is a metric space that has compact completion and is  homeomorphic to a domain in $\sphere$.

\begin{corollary} \label{good map} Suppose that $X$ satisfies the $\rm{LLC}_1$ condition or is a domain in $\sphere$. Then there is a continuous surjection $h \colon \ovl{X} \to \sphere$ such that $h|_{X}$ is a homeomorphism onto a domain $\Omega$ with totally disconnected complement. Moreover, the map $h$ is constant on each boundary component $E\in \mathcal{C}(X)$, and for any $\ep>0$, there is $\ep' > 0$ such that 
\begin{equation}\label{good map estimate}h\inv(B_{\sphere}(h(E),\ep')) \subeq \nbhd_{\ovl{X}}(E,\ep).\end{equation}
Finally, $h$ induces a homeomorphism from $\mathcal{C}(X)$ to $\partial\Omega$. 
\end{corollary}

\begin{proof}  We address only the case that $X$ satisfies the $\rm{LLC}_1$ condition.  We have assumed that $X$ is homeomorphic to a domain in $\sphere$.  Hence by Proposition~\ref{point boundary}, there is a homeomorphism $h \colon X \to \Omega$, where $\Omega \subeq \sphere$ is a domain with totally disconnected complement.  

By Theorem \ref{ends components summary} and Remark \ref{ends are topological}, there is a homeomorphism $\phi_0 \colon \mathcal{C}(X) \to \mathcal{E}(\Omega)$ with the property that a Cauchy sequence $\{x_i\} \in \mathcal{U}(X)$ converges to a point of the boundary component $E \in \mathcal{C}(X)$ if and only if $\{h(x_i)\}$ represents the end $\phi_0(E)$.  Moreover, Proposition \ref{point ends} provides a homeomorphism $\phi_1 \colon \mathcal{E}(\Omega) \to \partial{\Omega}$ with the property that a sequence $\{y_i\} \in \mathcal{U}(X)$ represents the end $E \in \mathcal{E}(\Omega)$ if and only if it converges to the boundary point $\phi_1(E) \in \partial{\Omega}$. We define the extension of $h$ to $\partial{X}$ by setting $h(x) = \phi_1 \circ \phi_0(E)$, where $E \in \mathcal{C}(X)$ is the boundary component containing $x \in \partial{X}$.  The naturality properties of $\phi_0$ and $\phi_1$ ensure that $h \colon \ovl{X} \to \sphere$ so defined is continuous.  As $h|_{X}$ is a homeomorphism onto $\Omega$ and $\partial{\Omega} = \sphere\bslash \Omega$, the definitions show that the extended map is a surjection.  

Now, let $E \in \mathcal{C}(X)$ and $\ep>0$.  By construction (or from the fact that $h|_{X}$ is a homeomorphism and $\sphere \bslash \Omega$ is totally disconnected), the set $h(E)$ consists of a single point in $\partial{\Omega}$. Suppose that there is no $\ep' >0$ such that \eqref{good map estimate} holds. Then there is a sequence of points 
$$x_n  \in \ovl{X} \bslash \nbhd_{\ovl{X}}(E,\ep)$$
such that $\{h(x_n)\}$ converges to $h(E) \in \partial{\Omega}$. By compactness and the fact that $h|_{X}$ is a homeomorphism onto $\Omega$, the sequence $\{x_n\}$ has a limit point $x \in \partial{X}\bslash \nbhd_{\ovl{X}}(E,\ep)$.  This means that $x$ lies in some boundary component $F \neq E$.  However, the continuity of $h$ implies that $h(F)=h(x)=h(E)$, which contradicts the fact that $\phi_1 \circ \phi_0$ is injective.

Corollary \ref{ends components summary} implies that $\mathcal{C}(X)$ is naturally  homeomorphic to $\mathcal{E}(X)$, and Remark \ref{ends are topological} shows that $\mathcal{E}(X)$ is naturally homeomoprhic to $\mathcal{E}(\Omega)$. Proposition \ref{point ends} now yields the final statement of the theorem.
\end{proof}

As an application of Proposition \ref{point boundary}, we prove Proposition \ref{ALLC rel sep circ}, which relates the $\rm{ALLC}$ condition and the relative separation of boundary components, and Proposition \ref{LLC gives ALLC}, which improves the $\rm{LLC}$ condition to the $\rm{ALLC}$ condition.

\begin{proof}[Proof of Proposition \ref{ALLC rel sep circ}] Let $\Omega$ be a circle domain, and suppose that there is a number $c >0$ such that the components $\{E_i\}_{i \in I}$ of $\partial \Omega$ satisfy 
\begin{equation}\label{rel sep in circ proof}\bigtriangleup(E_i,E_j) \geq c\end{equation}
whenever $i \neq j \in I$.  Fix $\Lambda \geq 1$ so large that $2c\inv < 2\Lambda^2-1$. We will show that $\Omega$ is $\Lambda$-$\rm{ALLC}$. Let $p \in \Omega$ and $r>0$. It suffices to show that $$A_\Omega(p,r/\Lambda,2\Lambda r)$$ 
is connected.   

Let $h \colon \ovl{\Omega} \to \sphere$ be the continuous surjection provided by Corollary \ref{good map}.  The complement of $h(\Omega)$ is a compact and totally disconnected set, and hence has topological dimension $0$ \cite[Section II.4]{Hurewicz}. The definition of $\Lambda$ and \eqref{rel sep in circ proof} guarantee that there is at most one index $i \in I$ such that $E_i$ intersects both $\ovl{B}_\sphere(p,r/\Lambda)$ and $\sphere \bslash B(p,2\Lambda r).$  This implies that $h(A_\Omega(p,r/\Lambda,2\Lambda r))$ is the complement, in the domain bounded by two Jordan curves that touch at no more than one point, of a set of topological dimension $0$.  It is therefore connected \cite[Theorem IV.4]{Hurewicz}. See Figure~\ref{CircALLC}.

\begin{figure}[h]
\begin{center}
\psfrag{p}{$p$}
\psfrag{r}{$r/\Lambda$}
\psfrag{R}{$2\Lambda r$}
\psfrag{h}{$h$}
\psfrag{hp}{$h(p)$}
\includegraphics[width=.75\textwidth]{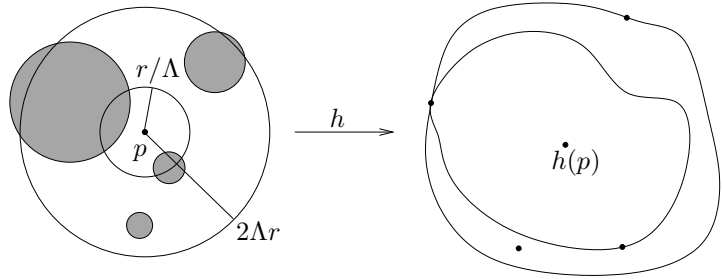}
\caption{The $\rm{ALLC}$ property of circle domains.}
\label{CircALLC}
\end{center}
\end{figure}
 
We leave the converse statement as an exercise for the reader, as it is not needed in this paper. 
\end{proof}

In the proof of Proposition \ref{LLC gives ALLC}, we will use the following separation theorem of point-set topology \cite[Section V.9]{Hurewicz}. 

\begin{theorem}[Janiszewski]\label{Janiszewski theorem} Suppose that $A$ and $B$ are closed subsets of $\sphere$ such that $\card(A \cap B) \leq 1$. If $y$ and $z$ are points of $\sphere$  that lie in the same component of $\sphere \bslash A$ and in the same component of $\sphere \bslash B$, then $y$ and $z$ lie in the same component of $\sphere \bslash (A \cup B)$. 
\end{theorem}

\begin{lemma}\label{boundary janiszewski} Let $(X,d)$ be a metric space homeomorphic to a domain in $\sphere$ that satisfies the $\rm{LLC}_1$-condition, and let $A$ and $B$ be disjoint closed subsets of $X$. Let $\mathcal{C}$ be the collection of components that intersect $\ovl{A} \cap \ovl{B}$, and assume that $\card \mathcal{C} \leq 1$. If $u, v \in X$ are in the same component of $X\bslash A$ and in the same component of $X\bslash B$, then they are in the same component of $X \bslash (A \cup B)$.  
\end{lemma}

\begin{proof} Let $h \colon \ovl{X} \to \sphere$ be the continuous surjection provided by Corollary \ref{good map}.  In particular $h|_X$ is a homeomorphism onto a domain $\Omega=\sphere \bslash T$, where $T$ is a closed totally disconnected set. Moreover, $h$ maps each component of $\partial{X}$ to a distinct point of $T$. Hence, our assumptions imply that $h(\ovl{A})$ and $h(\ovl{B})$ are closed subsets of $\sphere$ such that $h(\ovl{A}) \cap h(\ovl{B})$ is either empty or a single point of $T$.  Set 
$$\til{A} = h(\ovl{A}) \ \text{and} \ \til{B}=h(\ovl{B}) \cup T.$$
Then $\til{A}$ and $\til{B}$ are again closed subsets of $\sphere$ whose intersection is either empty or a single point, and 
$$\sphere \bslash \til{A} \supeq \Omega \bslash h(A) \ \text{and}\ \sphere \bslash \til{B} =\Omega \bslash h(B).$$
Thus, if $u$ and $v$ are points of $X$ that are in the same component of $X\bslash A$ and in the same component of $X \bslash B$, then $h(u)$ and $h(v)$ are in the same component of $\sphere \bslash \til{A}$ and in the same component of $\sphere \bslash \til{B}$.  Janiszewski's Theorem now implies that $h(u)$ and $h(v)$ are in the same component of 
$$\sphere \bslash (\til{A} \cup \til{B}) = \Omega \bslash (h(A) \cup h(B)).$$
Since $h|_X$ is a homeomoprhism onto $\Omega$, this yields the desired result.
\end{proof}

\begin{proof}[Proof of Proposition \ref{LLC gives ALLC}] By Lemma \ref{ALLC r 2r}, it suffices to consider a point $x \in X$ and radius $r>0$, and suppose that $y$ and $z$ are points of $A(x,r,2r)$. Denote by $\del>0$ the minimum distance between components of $\partial{X}$, and setF $s = \diam X/\del$.  By Proposition \ref{Better LLC} we may assume that $X$ satisfies the $\lambda$-$\til{\rm{LLC}}$ condition for some $\lambda \geq 1$.

We first assume that $r < \del/(4\lambda)$. Set
$$A = S_X(x,2\lambda r) \ \text{and}\ B=\ovl{B}_X(x,r/(2\lambda)).$$
The $\lambda$-$\til{\rm{LLC}}$-condition implies that $y$ and $z$ lie in the same component of $X \bslash A$ (namely, the component containing $x$) and in the same component of $X \bslash B$.  The restriction on $r$ implies that there is at most one component $E$ of $\partial{X}$ that satisfies $E \cap (\ovl{A} \cup \ovl{B}) \neq \emptyset.$  Hence, by Lemma \ref{boundary janiszewski}, the points $y$ and $z$ are in the same component of $X \bslash (A \cup B)$.  Since $(X,d)$ is locally path connected, this implies that they are contained in a continuum in $A(x,r/(2\lambda),2\lambda r).$

Since $A(x,r,2r)$ is empty if $r>\diam X$, we may now assume that $\del/(4\lambda) \leq r \leq \diam X.$ The $\til{LLC_1}$ condition provides an embedding $\gamma \colon [0,1] \to X$ such that $\gamma(0)=y$, $\gamma(1)=z$, and $\im \gamma \subeq B(x,2\lambda r)$.  If there is no $t \in [0,1]$ such that $\gamma(t) \in B(x,r/(16\lambda s))$, then the proof is complete.  Otherwise, let 
$$
\begin{aligned}
t_1&= \min\{t \in [0,1] : \gamma(t) \in S(x,r/(8\lambda s))\},\\ 
t_2&= \max\{t \in [0,1] : \gamma(t) \in S(x,r/(8\lambda s))\}.
\end{aligned}
$$
Since $\gamma$ is an embedding, $X$ is  connected, and $y,z \in A(x,r,2r)$, we see that $t_1 < t_2$.  As $r/(8\lambda s) < \del/(4\lambda)$, we may apply the first case considered above to the points $\gamma(t_1)$ and $\gamma(t_2)$, producing a continuum $\Gamma \subeq A(x,r/(16\lambda^2 s),r/(4s))$ that contains $\gamma(t_1)$ and $\gamma(t_2)$.  Now, the continuum
 $$\gamma([0,t_1]) \cup \Gamma \cup \gamma([t_2,1]) \subeq A(x,r/(16\lambda^2 s), 2\lambda r)$$
contains $y$ and $z$. \end{proof}

\section{Crosscuts}\label{crosscuts}

In this section, we assume that $X$ is a metric space that has compact completion, is homeomorphic to a domain in $\sphere$, and satisfies the $\lambda$-$\rm{LLC}_1$ condition for some $\lambda \geq 1$.

A \textit{crosscut} is an embedding $\gamma \colon [0,1] \to \ovl{X}$ such that $$\gamma([0,1]) \cap \partial{X} = \gamma(0) \cup \gamma(1).$$
Note that if $\gamma$ is a crosscut, then $\gamma|_{(0,1)}$ is a proper embedding.   If $\gamma(0)$ and $\gamma(1)$ are distinct points that belong to the same component $E$ of $\partial{X}$, then we say that $\gamma$ is an \emph{$E$-crosscut}.

The following proposition shows that under the assumption of the $\rm{LLC}_1$ condition, $E$-crosscuts behave as they do in the case of a circle domain in $\sphere$.  

\begin{proposition}\label{crosscut components}  Let $E$ be a component of $\partial{X}$, and let $\gamma$ be an $E$-crosscut. Then $X \bslash \im\gamma$ has precisely two components. Denoting these components by $U$ and $V$, it holds that:
\begin{itemize}
\item[(i)] for every $\ep>0$, there is a closed set $K \subeq X$ with $\dist(E,K)>0$ such that $U\bslash K$ is a connected subset of $\nbhd_{\ovl{X}}(E,\ep) \cap U$, and the same statement is valid for $V$,
\item[(ii)] the sets $\ovl{U} \bslash \im\gamma$ and $\ovl{V} \bslash \im\gamma$ are the components of $\ovl{X} \bslash \im\gamma$, and $\ovl{U} \cap \ovl{V} = \im \gamma$.
\item[(iii)] the sets $\ovl{U} \cap E$ and $\ovl{V} \cap E$ are connected.  
\end{itemize}
\end{proposition}

\begin{proof} By Proposition \ref{Better LLC}, we may assume that $X$ in fact satisfies the $\lambda$-$\til{\rm{LLC}}_1$ condition. 

Let $h \colon \ovl{X} \to \sphere$ be the continuous surjection provided by Corollary~\ref{good map}, and let $\Omega=h(X)$.  Since $h|_{X}$ is a homeomorphism, the map $h \circ \gamma \colon (0,1) \into \Omega$ is a proper embedding. As $h(E)$ is a single point and $h$ is continuous, 
\begin{equation}\label{cross cut limit} \lim_{t\to 0} h \circ \gamma(t)=h(E)=\lim_{t\to 1} h \circ \gamma(t). \end{equation}

This implies that the continuous map $h\circ \gamma \colon [0,1] \to \Omega \cup \{h(E)\}$ defines a Jordan curve.  The Jordan curve theorem states that $\sphere \bslash\im (h\circ \gamma)$ consists of two disjoint domains $\til{U}$ and $\til{V}$, each homeomorphic to $\reals^2$, that have common boundary $\im (h \circ \gamma)$.  Then $U := h\inv(\til{U} \cap \Omega)$ and $V := h\inv(\til{U} \cap \Omega)$ are disjoint non-empty open sets satisfying $U \cup V = X\bslash \im\gamma$.  As $h|_{X}$ is a homeomorphism, in order to show that $U$ and $V$ are the components of $X \bslash \im\gamma$, we need only show $\til{U} \cap \Omega$ and $\til{V} \cap \Omega$ are connected.   Since $\sphere \bslash \Omega$ is compact and totally disconnected, it has topological dimension $0$ \cite[Section II.4]{Hurewicz}.  Thus $\til{U} \cap \Omega$ homeomorphic to the complement in $\reals^2$ of a set of topological dimension $0$, and hence is connected \cite[Theorem IV.4]{Hurewicz}. The same proof applies to $\til{V} \cap \Omega$.

We procede to the proof of (i).  Let $\ep>0$.  By Corollary \ref{good map}, we may find $\ep'>0$ such that 
\begin{equation}\label{good map recap} h\inv(B_{\sphere}(h(E),\ep')) \subeq \nbhd_{\ovl{X}}(E,\ep).\end{equation}
As $\im (h \circ\gamma)$ is a Jordan curve, by Schoenflies'  theorem there is a homeomorphism $H \colon \sphere \to \sphere$ such that $H(\til{U})$ is a standard ball in $\sphere$ whose boundary contains the point $H \circ h(E)$. By the continuity of $H\inv$ and \eqref{good map recap} we may find $\ep''>0$ so small that 
\begin{equation}\label{H estimate} h\inv \circ H\inv (B_{\sphere}(H\circ h(E),\ep'')) \subeq \nbhd_{\ovl{X}}(E,\ep).\end{equation}
The set 
$$B_{\sphere}(H\circ h(E),\ep'') \cap H(\til{U})$$
is the non-empty intersection of two standard balls in $\sphere$ and hence is itself homeomorphic to $\reals^2$.  As before, \cite[Section II.4 and Theorem IV.4]{Hurewicz} imply that 
$$B_{\sphere}(H \circ h(E), \ep'') \cap H(\til{U}) \cap H(\Omega)$$ 
is connected.  It now follows from \eqref{H estimate}, the fact that $h|_{X}$ is a homeomorphism onto $\Omega$, and the definition of $U$ that 
$$h\inv \circ H\inv (B_{\sphere}(H\circ h(E),\ep'')) \cap U$$ 
is a connected subset of $\nbhd_{\ovl{X}}(E,\ep)$.  We set $$K = X \bslash (h\inv \circ H\inv (B_{\sphere}(H\circ h(E),\ep''))).$$  The continuity of $H \circ h$ now shows that $\dist(E,K)>0$.  An analogous proof applies to $V$. 

We next address (ii). It follows from the definitions that $\ovl{U} \bslash \im\gamma$ and $\ovl{V} \bslash \im\gamma$ are  non-empty, closed in $\ovl{X} \bslash \im\gamma$, and satisfy
$$(\ovl{U} \bslash \im\gamma) \cup (\ovl{V} \bslash \im \gamma) = \ovl{X} \bslash \im\gamma.$$
Moreover, as $\ovl{U}\bslash \im\gamma$ is the topological closure in $\ovl{X}\bslash \im \gamma$ of the connected set $U$, it is connected. Similarly, $\ovl{V}\bslash \im \gamma$ is connected. Hence $\ovl{U}\bslash \im\gamma$ and $\ovl{V}\bslash \im \gamma$ are the components of $\ovl{X}\bslash \im \gamma$.  

As the common boundary of  $\til{U}$ and $\til{V}$ is $\im h\circ \gamma$, we see that $\ovl{U} \cap \ovl{V} \supeq \im \gamma$.  Suppose there is a point $z \in \ovl{X} \bslash \im\gamma$ that is an accumulation point of both $U$ and $V$.  Let $0<\ep<\dist(z,\im\gamma)/\lambda$.  By assumption we may find points $u \in U \cap B_{\ovl{X}}(z,\ep)$ and $v\in V \cap B_{\ovl{X}}(z,\ep)$.  The $\lambda$-$\til{\rm{LLC}}_1$ condition now implies that $u$ and $v$ can be connected in $X\bslash\im\gamma$, a contradiction.  Hence $\ovl{U} \cap \ovl{V} = \im \gamma$.  

To prove (iii), we show that $\ovl{U} \cap E$ is connected; the corresponding statement for $V$ is proven in the same way.  If $\ovl{U} \cap E$ is not connected, then we may find disjoint, non-empty, and compact sets $A$ and $B$ such that $A \cup B = \ovl{U} \cap E$.  Fix $0<\ep<\dist(A,B)/2$.  We first claim that there is a number $\del>0$ such that 
$$\nbhd_{\ovl{X}}(E,\del) \cap U \subeq \nbhd_{\ovl{X}}(A \cup B,\ep)  \cap U.$$
If this claim is false, then for every $n \in \nats$ there are points $u_n \in U$ and $x_n \in E$ such that $d(x_n,u_n) < 1/n$ and $\dist(u_n, A \cup B) \geq \ep$.  Since $E$ is compact, there is a subsequence of $\{x_n\}$ that converges to a point $x \in E$.  It follows that the corresponding subsequence of $\{u_n\}$ converges to $x$ as well. Hence $x \in \ovl{U} \cap E$ but $\dist(x, A \cup B) \geq \ep$, a contradiction.  This proves the claim.

By (i), there is a closed set $K \subeq X$ such that $\dist(E,K)>0$ and $U \bslash K$ is a connected subset of $\nbhd_{\ovl{X}}(E,\del) \cap U,$ and hence, by the claim, of $\nbhd_{\ovl{X}}(A \cup B,\ep)  \cap U.$  However, since $\dist(E,K)>0$, we may find points 
$$a \in \nbhd_{\ovl{X}}(A,\ep)  \cap (U \bslash K) \mand b \in \nbhd_{\ovl{X}}(B,\ep)  \cap (U \bslash K).$$
This is a contradiction since $\dist(A,B) >2\ep$.  
\end{proof}

\section{Uniformization of the boundary components}\label{boundary uniformization section}
In this section, we assume that $X$ is a metric space that has compact completion, is homeomorphic to a domain in $\sphere$, and satisfies the full $\lambda$-$\til{\rm{LLC}}$ condition for some $\lambda \geq 1$. 

We prove that each boundary component $E \in \mathcal{C}(X)$ with at least two points is homeomorphic to $\mathbb{S}^1$ and satisfies the $\lambda'$-$\rm{LLC}$ condition, for some $\lambda'\geq 1$ depending only on $\lambda$. To do so, we employ a recognition theorem of point set topology: a metric space is homeomorphic to $\mathbb{S}^1$ if and only if it is a locally connected continuum such that removal of any one point results in a connected space, while the removal of any two points results in a space that is not connected \cite{Wilder}. This section adapts \cite[Section 4]{QSPlanes} to our setting; we omit certain proofs that need little or no translation.

\begin{proposition}\label{minus 1}  Let $E$ be a component of $\partial{X}$, and let $p \in E$.  Then $E \bslash \{p\}$ is connected.
\end{proposition}

\begin{proof} As a single point set and the empty set are connected, we may assume that $E$ has at least three points. Let $x$ and $y$ be arbitrary distinct points of $E \bslash \{p\}$. It suffices to show that there is a connected subset of $E\bslash \{p\}$ that contains $x$ and $y$.  By the $\til{\rm{LLC}}_1$-condition, there is an $E$-crosscut $\gamma$ with $\gamma(0)=x$ and $\gamma(1)=y$.  By Proposition \ref{crosscut components} (ii), there is a unique component $U$ of $X\bslash \im\gamma$ such that $\ovl{U} \cap E$ does not contain the point $p$.  Then $x$ and $y$ are contained $\ovl{U} \cap E$, which by Proposition~\ref{crosscut components}~(iii) is a connected subset of $E\bslash \{p\}$.   
\end{proof}

\begin{proposition}\label{minus 2} Let $E$ be a component of $\partial{X}$ with at least two points.  If $p,q \in E$ are distinct points, then $E\bslash \{p,q\}$ is not connected.  
\end{proposition}

\begin{proof} See \cite[Proposition 4.11]{QSPlanes}. Here, Proposition \ref{crosscut components} plays the role of \cite[Lemma 4.10]{QSPlanes}. \end{proof}

We now consider the local connectivity of boundary components of $X$.  We will need a few technical lemmas.  

\begin{lemma}\label{simple modification} Let $E$ be a component of $\partial{X}$, and suppose that $\gamma$ and $\gamma'$ are $E$-crosscuts with the property that there is a compact interval $I \subeq (0,1)$ such that $\gamma(t)=\gamma'(t)$ for all $t \notin I$.  Then there is a closed subset $K \subeq X$ with $\dist(E,K)>0$ such that if points $p$ and $q$ in $\ovl{X} \bslash K$ are in a single component of $\ovl{X} \bslash \im\gamma$, then they are in a single component of $\ovl{X} \bslash \im\gamma'$.\end{lemma}

\begin{proof} See \cite[Lemma 4.12]{QSPlanes}.
\end{proof}

We will briefly need a notion of transversality.  Let $Y$ be a topological space homeomorphic to a domain in $\sphere$, and let $\alpha,\beta \colon (0,1) \to Y$ be embeddings.  Given $x \in \im \alpha \cap \im\beta$, we say that $\alpha$ and $\beta$ \emph{intersect transversally at $x$} if there is an open neighborhood $U$ of $x$ and a homeomorphism $h \colon U \to \reals^2$ such that $h(U \cap \im \alpha)$ is the $x$-axis and $h(U \cap \im\beta)$ is the $y$-axis. 

\begin{lemma}\label{four points} Let $a$, $b$, $p$, and $q$ be distinct points on a Jordan curve $\alpha \subeq \sphere$, and let $U \subeq \sphere$ be a simply connected domain with boundary $\alpha$. If $V$ is any open subset of $\ovl{U}$ containing $\alpha$, then there are embeddings $\alpha_{ab} \colon [0,1] \to V$ and $\alpha_{pq} \colon [0,1] \to V$ connecting $a$ to $b$ and $p$ to $q$ respectively, such that either $\im \alpha_{ab} \cap \im \alpha_{pq} = \emptyset$, or $\alpha_{ab}$ and $\alpha_{pq}$ have a single intersection, and that intersection is in $U$ and is transverse. 
\end{lemma}

\begin{proof}  By the Sch\"{o}nflies theorem, there is a homeomorphism $H \colon \sphere \to \sphere$ such that $H(\alpha)$ is a round circle.  Then $H(V)$ is open in $H(\ovl{U})$ and it contains a round annulus that has $H(\alpha)$ as a boundary component.  Clearly $H(a)$, $H(b)$, $H(p)$, and $H(q)$ may be connected as desired inside this annulus.  Taking inverse images under $H$ now yields the desired result.
\end{proof}

\begin{lemma}\label{crossing} Let $a$, $b$, $p$, and $q$ be distict points of a component $E$ of $\partial{X}$, and let $\gamma_{ab}$ and $\gamma_{pq}$ be $E$-crosscuts connecting $a$ to $b$ and $p$ to $q$, respectively.  If $p$ and $q$ are contained in a single component of $\ovl{X} \bslash \gamma_{ab}$, then $a$ and $b$ are contained in a single component of $\ovl{X} \bslash \gamma_{pq}$.
\end{lemma}

\begin{proof}The fact that the points $a$, $b$, $p$, and $q$ are all distinct implies that $K = \im\gamma_{ab}\cap \im\gamma_{pq}$ is a compact subset of $X$. If $K$ is empty, then $\gamma_{ab}$ connects $a$ to $b$ without intersecting $\gamma_{pq}$, as desired.  Hence we may assume that $K \neq \emptyset$.  

By Corollary \ref{good map}, there is a continuous surjection $h \colon \ovl{X} \to \sphere$ where $h|_{X}$ is a homeomorphism onto a domain $\Omega$ in $\sphere$ with totally disconnected complement.

Let $\Omega_1\subeq \Omega_2 \subeq \hdots$ be the exhaustion of $\Omega$ guaranteed by Proposition \ref{good exhaustion}.  Since $a$, $b$, $p$, and $q$ are all elements of the same boundary component $E$, for each $i \in \nats$ there is a single simply connected component $U_i$ of $\sphere\bslash \Omega_i$ containing $h(\{a, b, p, q\})$.  
Fix $i \in \nats$ so large that $h(K) \subeq \Omega_i$.  Then we may find parameters $t_1,t_2,s_1,s_2 \in (0,1)$ such that 
$$t_1 = \min\{t\in [0,1]: h\circ \gamma_{ab}(t) \in \partial U_i \},$$
$$t_2 = \max\{t\in [0,1]:h\circ \gamma_{ab}(t) \in \partial U_i \},$$
$$s_1 = \min\{s\in [0,1]:h\circ \gamma_{pq}(t) \in \partial U_i \},$$ 
$$s_2 = \max\{s\in [0,1]:h\circ \gamma_{pq}(t)\in \partial U_i \}.$$

Since $\Omega_i$ has only finitely many boundary components and $h(K)$ is a compact subset of $\Omega_i$, we may find a relatively open neighborhood $V$ of $\partial{U_i}$ in $\sphere \bslash U_i$ such that $V \subeq \Omega_i \bslash h(K)$.  By Lemma \ref{four points}, there are embeddings $\alpha_{ab} \colon [0,1] \to V$ and $\alpha_{pq} \colon [0,1] \to V$ connecting $h\circ \gamma_{ab}(t_1)$ to $h \circ \gamma_{ab}(t_2)$ and $h\circ \gamma_{pq}(s_1)$ to $h\circ \gamma_{pq}(s_2)$ respectively, such that either $\im \alpha_{ab} \cap \im \alpha_{pq} = \emptyset$, or $\alpha_{ab}$ and $\alpha_{pq}$ have a single transversal intersection.  

Let $\til{\gamma}_{ab}$ be the path defined by concatenating $\gamma_{ab}|_{[0,t_1]}$, $h\inv \circ \alpha_{ab}$, and $\gamma_{ab}|_{[t_2,1]}.$ Similarly define $\til{\gamma}_{pq}$.  Then either $\im \til{\gamma}_{ab}$ and $\im \til{\gamma}_{pq}$ are disjoint, or they have a single intersection, and that intersection is transversal and located in $X$.  In the former case, $a$ and $b$ are in a single component of $\ovl{X} \bslash \til{\gamma}_{pq}$, and so Lemma \ref{simple modification} provides the desired result. In the latter case, the transversality and Proposition \ref{crosscut components} (ii) imply that $p$ and $q$ are not contained in a connected subset of $\ovl{X}\bslash \til{\gamma}_{ab}$. Lemma \ref{simple modification} shows that this is a contradiction.\end{proof}

\begin{proposition}\label{locally connected}
Let $E$ be a component of $\partial{X}$.  Then $E$ satisfies the $4\lambda^4$-$\rm{LLC}_1$ condition.  In particular, $E$ is locally connected. 
\end{proposition}

\begin{proof} We may assume that $E$ has at least two points. Let $p \in E$, and $r>0$.  It suffices to find a continuum $F$ such that 
$$B_{\ovl{X}}(p, r) \cap E \subeq F  \subeq B_{\ovl{X}}(p, 4\lambda^4 r) \cap E.$$

As $E$ itself is connected, we may assume that there is some point 
$$q\in E\bslash B_{\ovl{X}}(p, 4\lambda^4 r).$$
The $\til{\rm{LLC}}_1$ condition provides an $E$-crosscut $\gamma_{pq}$ connecting $p$ to $q$.  Let $U$ and $V$ be the components of $X\bslash \im{\gamma_{pq}}$, and set $A=\ovl{U}\cap E$ and $B=\ovl{V}\cap E$.  By Proposition \ref{crosscut components} (iii), $A$ and $B$ are connected.  As $\{p,q\} = A\cap B$ and $d(p,q) > 4\lambda^4 r$, we may find distinct points $a  \in A$ and $b \in B$ such that $d(p,a)= 2\lambda^2 r$ and  $d(p,b) =  2\lambda^2 r$.  The $\lambda$-$\til{\rm{LLC}}_1$ condition provides a crosscut $\gamma_{ab}$ connecting $a$ to $b$ with $\im{\gamma_{ab}} \subeq B_{\ovl{X}}(p,3\lambda^3 r).$  By Proposition \ref{crosscut components} (ii), there is a unique component $W$ of $X\bslash \im{\gamma_{ab}}$ with $p \in \ovl{W}\cap E$.  Set $F:=\ovl{W}\cap E$.  Applying Proposition \ref{crosscut components} (iii) again, we see that the set $F$ is connected.  


We first show that $F \subeq B_{\ovl{X}}(p, 4\lambda^4 r) \cap E.$  Suppose that there is a point $x \in F \bslash B_{\ovl{X}}(p, 4\lambda^4 r)$.  By the $\lambda$-$\til{\rm{LLC}}_2$ condition, there is a path connecting $x$ to $q$ without intersecting $B_{\ovl{X}}(p, 4\lambda^3 r)$.  This implies that $p$ and $q$ are in the same component of $\ovl{X}\bslash\gamma_{ab}$.  However, by Proposition \ref{crosscut components} (ii), the points $a$ and $b$ lie in different components of $\ovl{X}\bslash \im\gamma_{pq}$.  This contradicts Lemma \ref{crossing}. 

We now show that $B_{\ovl{X}}(p, r) \cap E \subeq F.$  Since $\gamma_{ab}$ is continuous, we may find parameters $0<t_a<1$ and $0< t_b < 1$ such that 
$$\diam(\gamma_{ab}([0,t_a]))\leq \lambda^2 r \quad \hbox{and} \quad \diam(\gamma_{ab}([t_b, 1]))\leq \lambda^2 r.$$
Set $a' = \gamma_{ab}(t_a)$ and $b' =\gamma_{ab}(t_b)$.  Then $a',b' \in X \bslash B_{\ovl{X}}(p, \lambda^2 r),$ and so the $\lambda$-$\til{\rm{LLC}}_2$ condition provides an embedding $\gamma_{a'b'}\colon [0,1] \to X$ such that $\gamma_{a'b'}(0)=a'$, $\gamma_{a'b'}(1)=b'$, and $\im{\gamma_{a'b'}} \subeq X\bslash B_{\ovl{X}}(p, \lambda r).$  The set 
$$S = \gamma_{ab}([0, t_a]) \cup \im{\gamma_{a'b'}} \cup \gamma_{ab}([t_b,1])$$
does not intersect $B_{\ovl{X}}(p, \lambda r),$ and is the image of a path in $\ovl{X}$.  Since the image of a path in $\ovl{X}$ is arc-connected, we may find a crosscut $\gamma'$ connecting $a$ to $b$ with $\im{\gamma'} \subeq S$.   Furthermore, we may find a compact interval $I \subeq (0,1)$ such that $\gamma'(t)=\gamma_{ab}(t)$ for all $t \in [0,1] \bslash I$.   Suppose that there is a point $x \in B_{\ovl{X}}(p, r) \cap E$ that is not contained in $F$.  Then $x$ and $p$ are in different components of $\ovl{X}\bslash\im{\gamma_{ab}}.$ By Lemma \ref{simple modification}, this implies that $x$ and $p$ are in different components $\ovl{X}\bslash\im{\gamma'}$. However, the $\lambda$-$\til{\rm{LLC}}_1$ condition shows that $x$ and $p$ may be connected by an arc contained in $B_{\ovl{X}}(p,\lambda r)$.  This is a contradiction.
\end{proof}

\begin{proof}[Proof of Theorem \ref{boundary uniformization}]  We suppose that $X$ is a metric space homeomorphic to a domain in $\sphere$, has compact completion, and satisfies the $\lambda$-$\rm{LLC}$ condition, $\lambda \geq 1$.  By Proposition \ref{Better LLC}, $X$ in fact satisfies the $\lambda'$-$\til{\rm{LLC}}$ condition for some $\lambda'$ depending only on $\lambda$.  Let $E$ be a component of $\partial{X}$ with at least two points.  As $\ovl{X}$ is compact, $E$ is a continuum.  Hence, Propositions \ref{minus 1}, \ref{minus 2}, and \ref{locally connected} along with the recognition theorem of \cite{Wilder} show that $E$ is a topological circle.  Proposition \ref{locally connected} shows that $E$ is $4\lambda'^4$-$\rm{LLC}_1$. The desired statement now follows from \cite[Proposition 4.15]{QSPlanes} and the characterization of quasicircles given in \cite{QS}. \end{proof}

\begin{remark}\label{E components} Let $E$ be a component of $\partial{X}$, let $\gamma$ be an $E$-crosscut, and let $U$ and $V$ be the components of $X \bslash \im \gamma$. By Proposition \ref{crosscut components} (ii) and (iii), the sets $\ovl{U} \cap E$ and $\ovl{V} \cap E$ are connected and have intersection $\{\gamma(0),\gamma(1)\}$.  Since Theorem \ref{boundary uniformization} implies that $E$ is a topological circle, we may conclude that $(\ovl{U} \cap E) \bslash \im\gamma$ and $(\ovl{U} \cap E) \bslash \im \gamma$ are the components of $E\bslash \im \gamma$. 
\end{remark}

\section{Topological uniformization of the completion}\label{completion uniformization section}

In this section, in which the notation and assumptions are as in the previous section, we give the following topological uniformization for the completion of $X$, at least in the case of finitely many boundary components.  

\begin{theorem}\label{closure uniformization} Let $N \in \nats$, and assume that $\partial{X}$ has $N$ components, all of which are non-trivial.  Then the completion $\ovl{X}$ is homeomorphic to the closure of a circle domain that has $N$ boundary components, all of which are non-trivial.
\end{theorem}

\begin{remark}\label{RP2} The homeomorphism type of a metric space $X$ and the homeomorphism type of the boundary $\partial{X}$ do not, in general, determine the homemorphism type of the completion $\ovl{X}$.  The following example demonstrates this. The real projective plane $\reals P^2$, formed by identifying antipodal points of $\sphere$, can be metrized as a subset of $\reals^4$ endowed with the standard metric. Let $X\subeq \reals P^2 \subeq \reals^4$ be the image of the sphere minus the equator under the indentification of antipodal points.  Then $X$ is homeomorphic to the disk, and $\partial{X}$ is homeomorphic to the image of the equator under the indentification of antipodal points, i.e., it is homeomorphic to the circle. However, the completion $\ovl{X}$ is homeomorphic to all of $\reals P^2$, and not the closed disk. The $\rm{LLC}$ condition prevents this phenomena from occurring in the setting we are most interested in.
\end{remark}

The proof of Theorem \ref{closure uniformization} relies on the following topological characterization of the closed disk, due to Zippin \cite{Zippin}.  Here, given a topological space $Y$, an embedding $\alpha \colon [0,1] \to Y$ is said to \emph{span} a subset $J$ of $Y$ if $\im \alpha \cap J = \{\alpha(0),\alpha(1)\}$ and $\alpha(0)\neq \alpha(1	)$.  Note that an embedding spans a component $E$ of the boundary $\partial{X}$ if and only if it is an $E$-crosscut.

\begin{theorem}[Zippin]\label{closed disk char} A locally connected and metrizable continuum $Y$ is homeomorphic to the closed disk $\ovl{\disk}$ if and only if $Y$ contains a topological circle $J$ with the following properties
\begin{itemize}
\item there is an embedding $\alpha \colon [0,1] \to Y$ that spans $J$,
\item for every embedding $\alpha \colon [0,1] \to Y$ spanning $J$, the set $Y \bslash \im \alpha$ is not connected,
\item for every embedding $\alpha \colon [0,1] \to Y$ spanning $J$ and every closed subset $I \subsetneq [0,1]$, the set $Y \bslash \alpha(I)$ is connected.
\end{itemize}
\end{theorem}

\begin{remark}\label{boundary is J} The homeomorphism constructed in the proof of Theorem \ref{closed disk char} maps the Jordan curve $J$ onto the unit circle. 
\end{remark}

\begin{lemma}\label{annular neighborhood lemma} Let $E \in \mathcal{C}(X)$ be an isolated and non-trivial component of $\partial{X}$, and let $h \colon \ovl{X} \to \sphere$ be the continuous surjection provided by Corollary \ref{good map}. Then for all sufficiently small $\ep>0$, the set
$$U =  h\inv(\ovl{B}_{\sphere}(h(E),\ep))$$
is homeomorphic to a closed annulus.
\end{lemma}

\begin{proof} Corollary \ref{good map} states that $h$ induces a homeomorphism from $\mathcal{C}(X)$ to $\partial(h(X))$. Since $E$ is isolated, it follows that $U \cap \partial{X}=E$ when $\ep$ is sufficiently small.

Fix such an $\ep$, and set $p = h(E) \in \sphere$ and $\alpha_0 =h\inv(S_\sphere(p,\ep)).$ Then $\alpha_0$ is a Jordan curve in $X$. Equip 
$$U \coprod (\sphere \bslash B_{\sphere}(p,\ep))$$
with the disjoint union topology, and consider the quotient topological space
$$Y = \left(U \coprod (\sphere \bslash B_{\sphere}(p,\ep)) \right) / x\sim h(x), \quad x \in \alpha_0.$$
The result will follow once it is shown that $Y$ is homeomorphic to the closed disk, with $\pi(E)$ corresponding the unit circle.  To do so, we employ Theorem \ref{closed disk char}. 

Let $\pi$ denote the usual projection map onto $Y$, and note that 
$$\pi|_{U} \ \text{and}\ \pi|_{\sphere \bslash B_{\sphere}(p_,\ep)}$$
are embeddings, and $\pi|_{\alpha_0}$ is a two-to-one map. After some effort, it can be seen that $Y$ is compact, second-countable, and regular. The $\til{\rm{LLC}}$-condition on $\ovl{X}$  also implies that  $Y$ is connected and locally path-connected. By Urysohn's metrization theorem, it is also metrizable. Moreover, as $h|_X$ is a homeomorphism, the set $U \bslash E$ is homeomorphic to $\ovl{B}_{\sphere}(p,\ep)\bslash \{p\}$.  Thus, elementary point-set topology shows that $Y\bslash \pi(E)$ is homeomorphic to $\sphere \bslash \{p\}$, i.e., to the plane.  

By Theorem \ref{boundary uniformization}, the set $\pi(E)$ is a topological circle in $Y$.  The existence of an embedding $\alpha \colon [0,1] \to Y$ that spans $\pi(E)$ follows from the $\til{\rm{LLC}}$ condition on $X$. We next check that any embedding $\alpha \colon [0,1] \to Y$ that spans $\pi(E)$ disconnects $Y$. Towards a contradiction, suppose that $Y\bslash \alpha$ is connected. Since $Y$ is locally path-connected and Hausdorff, it follows that $Y\bslash \alpha$ is arc-connected. 

We assume that the image of $\alpha$ intersects the circle $\pi(\alpha_0)$; the argument in the case that this does not occur is a simpler version of what follows. Set
$$t_0 = \min\{t \in [0,1]: \alpha(t) \in \pi(\alpha_0)\} \ \text{and}\ t_1 = \max\{t \in [0,1]: \alpha(t) \in \pi(\alpha_0)\}.$$
 Then $0<t_0 \leq t_1 < 1$.  Denote
\begin{align*} &p=\alpha(0) \in \pi(E),\\
						&q=\alpha(1) \in \pi(E),\\
						&p'=\alpha(t_0) \in \pi(\alpha_0),\\
						&q'=\alpha(t_1) \in \pi(\alpha_0),
						\end{align*}
and let $A$ be an arc of the topological circle $\pi(\alpha_0)$ with endpoints $p'$ and $q'$. It is possible that $A$ is a single point.  Let $\alpha' \colon [0,1] \to Y$ be an embedding such that $\alpha'(t)=\alpha(t)$ when $t \in [0,t_0] \cup [t_1,1]$, and such that $\alpha' \colon [t_0,t_1] \to Y$ is an embedding parameterizing $A$.  Then $\im(\alpha') \subeq \pi(U)$, and hence $\alpha'$ defines an $E$-crosscut $\til{\alpha}$ in $\ovl{X}$.  By Remark \ref{E components}, we may find points $x$ and $y$ in $E$ that are in different components of $\ovl{X} \bslash \til{\alpha}$.  Since 
$$\im \alpha \cap \pi(E) = \pi(\im \til{\alpha} \cap E),$$ 
the points $\pi(x)$ and $\pi(x)$ are in $Y\bslash \alpha$.  

Our assumption now provides an embedding $\gamma \colon [0,1] \to Y\bslash \alpha$ such that $\gamma(0)=\pi(x)$ and $\gamma(1)=\pi(y)$.  Since $Y\bslash \pi(E)$ is homeomorphic to the plane $\reals^2$, Sch\"onflies theorem implies there is a homeomorphism $\Phi \colon Y\bslash \pi(E) \to \reals^2$ that sends $\alpha|_{(0,1)}$ to the real line. Since $\pi(x)$ and $\pi(y)$ are in $\pi(E)$, the embedding $\Phi \circ \gamma|_{(0,1)}$ is proper. Moreover, its image does not intersect the real line. Let $B$ be a ball in $\reals^2$ that contains the compact set $\Phi(\pi(\alpha_0))$.  As $\Phi \circ \gamma|_{(0,1)}$ is proper, we may find points $s_0,s_1 \in (0,1)$ such that if $t \notin (s_0,s_1)$, then $\Phi \circ \gamma(t) \notin B$. From the geometry of $\reals^2$, we see that there is a path 
$$\beta \colon [s_0,s_1] \to \reals^2\bslash (B \cup \reals \times \{0\}).$$
Then $\im\beta$ does not intersect $\im(\Phi \circ \alpha').$ Let $\gamma'$ denote the concatenation of $\gamma|_{[0,s_0]}$, $\Phi\inv \circ \beta$, and $\gamma|_{[s_1,1]}.$  Then $\im \gamma'$ contains $\pi(x)$ and $\pi(y)$ and is contained in $\pi(U) \bslash \alpha'$.  It follows that $\im(\pi\inv (\gamma'))$ connects $x$ to $y$ inside of $\ovl{X} \bslash \til{\alpha}$, contradicting the definition of $x$ and $y$.  See Figure \ref{nbhdannulus}.

\begin{figure}[h]
\begin{center}
\psfrag{p}{$p$}
\psfrag{q}{$q$}
\psfrag{p'}{$p'$}
\psfrag{q'}{$q'$}
\psfrag{x}{$\pi(x)$}
\psfrag{y}{$\pi(y)$}
\psfrag{a}{$\alpha$}
\psfrag{pa0}{$\pi(\alpha_0)$}
\psfrag{pE}{$\pi(E)$}
\psfrag{g}{$\gamma$}
\psfrag{A}{$A$}
\psfrag{B}{$B$}
\psfrag{b}{$\beta$}
\psfrag{ph}{$\Phi$}
\psfrag{phg}{$\Phi(\gamma)$}
\psfrag{pha}{$\Phi(\alpha)$}
\psfrag{phA}{$\Phi(A)$}
\psfrag{php}{$\Phi(p')$}
\psfrag{phq}{$\Phi(q')$}
\includegraphics[width=.75\textwidth]{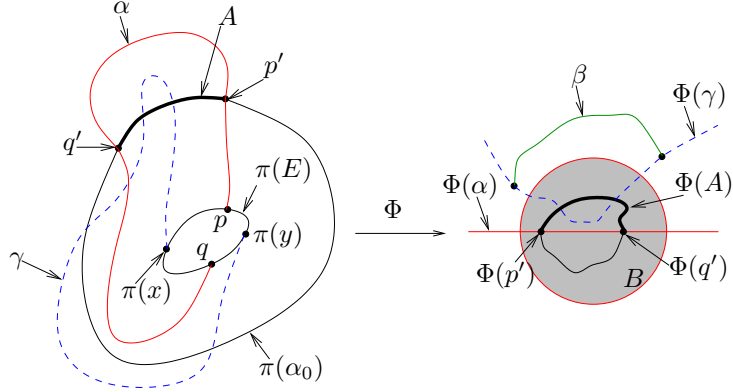}
\caption{The proof of Lemma \ref{annular neighborhood lemma}.}\label{nbhdannulus}
\end{center}
\end{figure}

Finally, we check that if $I \subsetneq [0,1]$ is closed, then $\alpha(I)$ does not separate the $Y$.  Since $Y$ is locally path connected, it suffices to show that any pair of points $x,y \in Y \bslash \pi(E)$ can be connected without intersecting $\alpha(I)$. This follows easily from the Sch\"onflies theorem.
\end{proof}

We will also need the following well-known topological statement, the proof of which is left to the reader.
\begin{lemma}\label{untwist} Let $\ep>0$ and $p \in \sphere$.  Given any homeomorphism $\phi \colon S_{\sphere}(p,\ep) \to  S_{\sphere}(p,\ep) $, there is a homeomorphism $\Phi \colon A_{\sphere}(p,\ep,2\ep) \to A_{\sphere}(p,\ep,2\ep)$ such that $\Phi|_{S_\sphere(p,\ep)}=\phi$ and $\Phi|_{S_\sphere(p,2\ep)}$ is the identity mapping.
\end{lemma}

Given a metric space $Y$ homeomorphic to a domain in $\sphere$, we denote by $\mathcal{N}(Y) \subeq \mathcal{C}(Y)$ the collection of non-trivial components of $\partial{Y}$, and by $\mathcal{I}(Y) \subeq \mathcal{N}(Y)$ the collection of non-trivial components of $\partial{Y}$ that are isolated as points in $\mathcal{C}(Y)$.

Theorem \ref{closure uniformization} is a special case of the following result. 

\begin{theorem}\label{isolate closure uniformization} Denote 
$$\til{X} =  X \cup \left(\bigcup_{E \in \mathcal{I}(X)} E\right).$$
Then there is a circle domain $\Omega' \subeq \sphere$ and a homeomorphism 
$$\til{h} \colon \til{X} \to \Omega' \cup \left(\bigcup_{F \in \mathcal{N}(\Omega')}F\right).$$
Moreover, $\Omega'$ and $\til{h}$ may be chosen so that $\til{h}$ induces a homeomorphism from $\mathcal{C}(X)$ to $\mathcal{C}(\Omega')$.
\end{theorem}

\begin{proof} Let $h \colon \ovl{X} \to \sphere$ denote the continuous surjection provided by Corollary \ref{good map}. Let $E \in \mathcal{I}(X)$. By Lemma \ref{annular neighborhood lemma}, there is a number $\ep_E > 0$ and a homeomorphism 
$$h_E \colon  h\inv(\ovl{B}_{\sphere}(h(E),\ep_E/2)) \to \ovl{A}_{\sphere}(h(E),\ep_E/4,\ep_E/2).$$
Denote the domain of $h_E$ by $U_E$, and set $V_E= h\inv(\ovl{B}_{\sphere}(h(E),\ep_E)).$
We may assume that the Jordan curve $\beta_E := h\inv(S_\sphere(h(E),\ep_E))$ is mapped onto $S_\sphere(h(E),\ep_E)$. Moreover, we may choose the number $\ep_E$ so small that 
$V_E \cap \partial{X}=E,$ and that the resulting collection $\{V_E\}_{E \in \mathcal{I}(X)}$ is pairwise disjoint. 

According to Lemma \ref{untwist}, for each $E \in \mathcal{I}(X)$, there is a homeomorphism $\Phi_E$ of $A_{\sphere}(h(E),\ep_E/2,\ep_E)$ to itself that agrees with $h_E \circ h\inv$ on $S_\sphere(h(E),\ep_E/2)$ and is the identity on $S_\sphere(h(E),\ep_E)$.  

Define
$$\Omega' = h(X) \bslash \left(\bigcup_{E \in \mathcal{I}(X)}\ovl{B}_{\sphere}(h(E),\ep_E/4)\right).$$
Then the collection of non-trivial boundary components of $\Omega'$ is given by $\mathcal{N}(\Omega') = \{S_{\sphere}(h(E),\ep_E/4)\}_{E \in \mathcal{I}(X)}.$ The map $\til{h} \colon \til{X} \to \Omega' \cup \left(\bigcup_{F \in \mathcal{N}(\Omega')}F\right)$ defined by 
$$\til{h}(x) = \begin{cases}
					h(x) & x \notin \bigcup_{E \in \mathcal{I}(X)} V_E, \\
					\Phi_E\circ h(x) & x \in V_E \bslash U_E,\\
					h_E(x) & x \in U_E\\
					\end{cases}$$
now yields the desired homeomorphism. See Figure \ref{twist}.

\begin{figure}[h]
\begin{center}
\psfrag{U}{$U_E$}
\psfrag{V}{$V_E$}
\psfrag{hE}{$h_E$}
\psfrag{h}{$h$}
\psfrag{hEp}{$h(E)$}
\psfrag{e}{$\ep$}
\psfrag{e2}{$\ep/2$}
\psfrag{e4}{$\ep/4$}
\psfrag{E}{$E$}
\includegraphics[width=.75\textwidth]{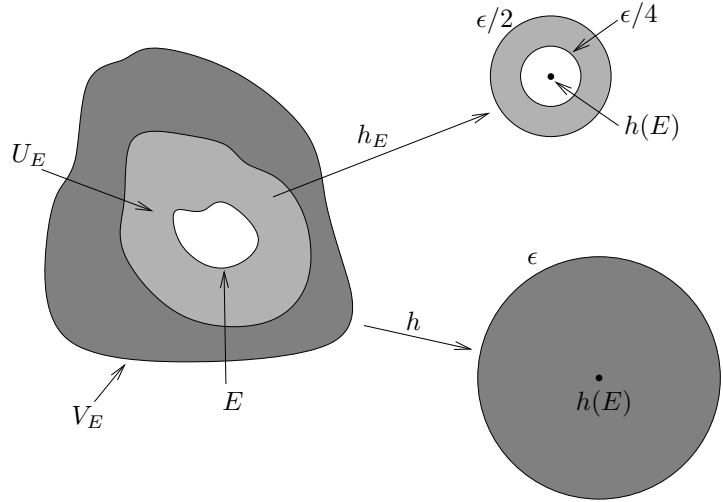}
\caption{Topological uniformization of the completion.}\label{twist}
\end{center}
\end{figure}

The final assertion follows from the construction and the fact that there is a natural homeomorphism from $ \mathcal{C}(h(X))$ to $\mathcal{C}(\Omega')$.
\end{proof}

\section{$\rm{ALLC}$ and porous quasicircles}\label{porosity section}

 A subset $Z$ of a metric space $(X,d)$ is \emph{$C$-porous}, $C \geq 1$, if for every $z \in Z$ and $0<r \leq \diam X$, there is a point $x \in X$ such that 
$$B\left(x, \frac{r}{C}\right) \subeq B(z,r) \bslash Z.$$
Porous subsets are small compared to the ambient space in a quantitative sense. See \cite[Section 5.8]{Fractured} and \cite[Lemma 3.12]{DoubConf} for a proof of the following well-known theorem.

\begin{theorem}\label{porous dim} Let $(X,d)$ be an Ahlfors $Q$-regular metric space, $Q>0$, and let $\Gamma$ be a subset of $X$. Then $\Gamma$ is $C$-porous for some $C \geq 1$ if and only $\Gamma$ is $(\alpha,C')$-homogeneous for some $\alpha<Q$ and $C' \geq 1$, quantitatively.
\end{theorem}

\begin{theorem}\label{porous quasicircles} Let $(X,d)$ be an $\rm{ALLC}$ metric space, and let $\Gamma \subeq X$ be a quasicircle.  Then $\Gamma$ is porous in $X$, quanitatively.
\end{theorem} 

\begin{proof} We suppose that $(X,d)$ is $\Lambda$-$\rm{ALLC}$, $\Lambda \geq 1$, and that $\Gamma \subeq X$ is a Jordan curve satisfing the following condition: there is $\lambda \geq 1$ such that for each pair of distinct points $x$ and $y$ on $\Gamma$
$$\diam(I) \leq \lambda d(x,y),$$
where $I$ is a component of $\Gamma\bslash \{x,y\}$ of minimal diameter. This condition is quantitatively equivalent to the assumption that $\Gamma$ is a quasicircle; see \cite{QS}.

Let $z \in \Gamma$ and $0 < r \leq \diam X$. We consider three cases.

\textit{Case 1: $0 < r < (\diam \Gamma)/(4\Lambda)$.}  We may find a point $w \in \Gamma$ such that $d(z,w) \geq 2\Lambda r$.  We may also find points $u,v \in \Gamma$ such that $\{z,v,w,u\}$ is cyclically ordered on $\Gamma$, $d(z,u)=r/(4\Lambda)=d(z,v)$, and if $J(z)$ is the component of $\Gamma \bslash \{u,v\}$ that contains $z$, then $J(z) \subeq B(z, r/(4\Lambda)).$ 

The $\Lambda$-$\rm{ALLC}$ condition on the space $X$ implies the existence of a continuum $\alpha \subeq  A(z,r/(8\Lambda^2),r/2)$ that contains $u$ and $v$.  Let $J(u)$ be the component of $\Gamma\bslash \{z,w\}$ containing $u$, and let $I(u) = J(u) \cup \{z,w\}.$ Define $J(v)$ and $I(v)$ similarly. We claim that $\dist(v, I(u)) \geq r/(8\lambda \Lambda).$  If not, then the $\lambda$-three point condition implies that either $z$ or $w$ is within a distance of $r/(8\Lambda)$ of $v$, which is not the case.  Now, the connectedness of $\alpha$ implies that there is a point $x \in \alpha$ such that 
$$\dist(x,I(u)) = \frac{r}{s}, \ \text{where} \ s=8(2\lambda +1)\Lambda^2.$$
Suppose that there is a point $y \in I(v)$ such that $d(x,y) < r/s$.  Then $\dist(y, I(u))<2r/s$, and hence the $\lambda$-three point condition implies that either $z$ or $w$ is within a distance $2\lambda r/s$ of $y$, and hence 
$$\dist(x,\{w,z\}) < \frac{2\lambda r}{s} + \frac{r}{s} = \frac{r}{8\Lambda^2}.$$
Combined with the facts that $x \in \alpha \subeq A(z,r/(8\Lambda^2),r/2)$ and $d(z,w)\geq 2\Lambda r$, this yields a contradiction. Hence $\dist(x,\Gamma) \geq r/s$, and so the fact that $d(x,z) < r/2$ implies
$$B\left(x, \frac{r}{s}\right) \subeq B(z,r) \bslash \Gamma.$$

\textit{Case 2: $8\diam \Gamma \leq r \leq \diam X$.} We may find a point $x \in X$ such that $d(x,z)=r/4$.  Since $\diam \Gamma \leq r/8$ and $z \in \Gamma$, we see that 
$$\dist(x, \Gamma) \geq d(x,z) - \diam\Gamma \geq r/8.$$
Hence $B(x,r/8) \subeq B(z,r) \bslash \Gamma.$ 

\textit{Case 3: $\diam \Gamma/(4\Lambda) \leq r < 8\diam \Gamma.$}  In this case,
$$\frac{r}{32 \Lambda} < \frac{\diam \Gamma}{4\Lambda}.$$
Thus, Case 1 implies that there is a point $x \in X$ such that 
$$B\left(x, \frac{r}{32\Lambda s}\right) \subeq B\left(z, \frac{r}{32\Lambda}\right) \bslash \Gamma \subeq B(z,r)\bslash \Gamma.$$
\end{proof}

%
We combine the results of this section with those of the previous sections in the following statement.

\begin{corollary}\label{bdry porosity ALLC} Let $(X,d)$ be a doubling metric space that is homeomorphic to a domain in $\mathbb{S}^2$, has compact completion, and satisfies the $\rm{ALLC}$ condition.  Then each component of the boundary $\partial{X}$ is a porous subset of $\ovl{X}$, quantitatively.
\end{corollary}

\begin{proof} By Lemma \ref{ALLC gives LLC}, the space $X$ is also $\rm{LLC}$, quantitatively. Theorem \ref{boundary uniformization} now implies that each component of the boundary $\partial{X}$ is a quasicircle, quantitatively.  Thus Theorem \ref{porous quasicircles} yields the desired result.
\end{proof}

The following theorem states that, up to bi-Lipschitz equivalence, having Assouad dimension strictly less than $2$ characterizes quasicircles in $\sphere$ among the class of all quasicircles \cite{HerronMeyer}.

\begin{theorem}[Herron--Meyer]\label{qcircle filling} Let $(\Gamma,d_\Gamma)$ be a metric circle.  The following statements are equivalent, quantitatively.
\begin{itemize}
\item $\Gamma$ is a quasicircle that is $(\alpha,C)$-homogeneous for some $1\leq \alpha <2$ and $C\geq 1$. 
\item $\Gamma$ is bi-Lipschitz equivalent to a quasicircle in $\sphere$.
\end{itemize}
\end{theorem}

\begin{remark}\label{Herron Meyer remark} Combined with the classical theory of planar quasiconformal mappings, Theorem \ref{qcircle filling} has the following consequence, which we will use in the proof of our main result. Let $\Gamma$ be a quasicircle that is $(\alpha,C)$-homogeneous for some $1\leq \alpha <2$ and $C\geq 1$. Then there is a domain $D_\Gamma \subeq \sphere$ such that  
\begin{itemize}
\item[(i)] $D_\Gamma$ is Ahlfors $2$-regular,
\item[(ii)] $D_\Gamma$ is quasisymmetrically equivalent to $\disk$,
\item[(iii)] the boundary $\partial{D}_{\Gamma}$ is bi-Lipschitz equivalent to $\Gamma$.
\end{itemize}
\end{remark}

\section{Gluing} \label{gluing section}

We now describe a process for gluing together metric spaces along bi-Lipschitz equivalent subsets. Parts of the basic construction maybe found in \cite{Burago}, and related deeper results are included in \cite{sewing}.

For the remainder of this section, we let $I$ be a possibly uncountable index set, which we extend by one symbol $0$ to create the index set $I_0=I \cup\{0\}$. We consider a collection $\{(X_i,d_i)\}_{i \in I_0}$ of compact metric spaces and a pairwise disjoint collection $\{E_i\}_{i \in I}$ of continua in $X_0$ such that there is a number $L \geq 1$ such that for each $i \in I$, there exists an $L$-bi-Lipschitz homeomorphism $f_i \colon E_i \to f_i(E_i) \subeq X_i$.

\subsection{The basic gluing construction}
We first consider the disjoint union 
$$\til{Z} = \coprod_{i \in I_0} X_i.$$
We then consider the set $Z$ obtained by gluing each space $X_i$ to $X_0$ via $f_i$, i.e., $Z$ is the quotient of $\til{Z}$ by the equivalence relation $\sim$ generated by the condition that for all $i \in I$, if $x \in E_i$, then $x \sim f_i(x)$. 

We wish to define a natural metric on $Z$.  To do so, we define an auxilliary distance function for points $z,w \in \til{Z}$ by 
$$\til{d}(z,w) = \begin{cases}
			d_i(z,w) & z,w \in X_i,\\
			\infty & \text{otherwise}.\\
			\end{cases}$$
We now define a distance function $d$ on $Z$ by setting, for all equivalence classes $a,b \in Z$,
$$d(a,b) = \inf \sum_{k=1}^n \til{d}(z_k,z'_k),$$
where the infimum is taken over all sequences $\mbf{z}=z_1,z'_1,\hdots,z_n,z'_n$ in $\til{Z}$ such that $z_1 \in a$, $z'_n \in b$, and if $n>1$, then $\pi(z'_k)=\pi(z_{k+1})$ for all $i=1,\hdots,n-1$. We say that such a sequence $\mbf{z}$ is an \emph{admissible sequence from $a$ to $b$} if the \emph{$\til{d}$-length of $\mbf{z}$}
$$l_{\til{d}}(\mbf{z})=\sum_{k=1}^n \til{d}(z_k,z'_k)$$
is finite, and if for any $k=1,\hdots, n-1$,
\begin{equation}\label{admissible} z'_k \neq z_{k+1}.\end{equation}
The triangle inequality implies that the infimum in the definition of $d(a,b)$ may be taken over all admissible sequences.  

\begin{proposition}\label{comparison prop} The distance function $d$ is a metric on the set $Z$, and if points $a$ and $b$ of  $Z$ have representatives $z_a$ and $z_b$ satisfying $\til{d}(z_a,z_b)<\infty$, then
\begin{equation}\label{comparison}\frac{\til{d}(z_a,z_b)}{L} \leq d(a,b) \leq \til{d}(z_a,z_b).\end{equation}
\end{proposition}

\begin{proof} Let $a,b \in Z$.  The definitions quickly imply that $d(a,b)=d(b,a)$, that $d(a,a)=0$, and that $d$ satisfies the triangle inequality.   

Before showing that $d(a,b)=0$ implies that $a=b$, we prove \eqref{comparison}. Let $z_a$ and $z_b$ be representatives of $a$ and $b$ respectively, such that $\til{d}(z_a,z_b) <\infty$. The second inequality in \eqref{comparison} follows from the fact that $z_a,z_b$ is an admissible sequence connecting $a$ to $b$.  Towards a proof of the first inequality, let $\mbf{z}=z_1,z'_1,\hdots,z_n,z'_n$ be an admissible sequence connecting $a$ to $b$.  We consider only the case that $z_a$ and $z_b$ are in $X_0$; the other cases are handled similarly. It suffices to assume that $z_1=z_a$ and $z'_n = z_b$ and to show that $l_{\til{d}}(\mbf{z}) \geq \til{d}(z_a,z_b)/L$. For each $k=1,\hdots,n$, let $i_k \in I_0$ be the index such that $z_k,z'_k \in X_{i_k}$.  Since $\mbf{z}$ is admissible, we may assume that if $i_k\neq 0$, then $z_k,z'_k \in f_{i_k}(E_{i_k}).$ Since each $f_{i_k}$ is an $L$-bi-Lipschitz mapping, the triangle inequality implies that 
\begin{align*} l_{\til{d}}(\mbf{z}) &= \sum_{k=1}^{n} d_{i_k}(z_k,z'_k)\\ &\geq \sum_{\{k:  i_k=0\}}d_0(z_k,z'_k) + \sum_{\{k:i_k\neq 0\}}\frac{d_0(f_{i_k}\inv(z_k),f_{i_k}\inv(z'_k))}{L} \\
&\geq \frac{d_0(z_a,z_b)}{L} =\frac{\til{d}(z_a,z_b)}{L},\end{align*}
as desired.

Now, suppose that $d(a,b)=0$. If there are representatives $z_a$ and $z_b$ of $a$ and $b$ respectively that satisfy $\til{d}(z_a,z_b)<\infty$, then \eqref{comparison} shows that $z_a=z_b$, and hence $a=b$.  Suppose no such representatives exist.  Then we may assume without loss of generality that $a$ has a representative $z_a \in X_{i}\bslash f_i(E_{i})$ for some $i \in I$, and that $b$ has no representative in $X_{i}$.  Then any admissible sequence from $a$ to $b$ has $\til{d}$-length at least $\dist_{d_{i}}(z_a,f_i(E_{i}))$.  Since $f_i(E_i)$ is compact, we conclude that $d(a,b)$ is positive, a contradiction. 
\end{proof}

\begin{remark}\label{different topologies remark} The metric $d$ defines a topology on $Z$, which we will refer to as \emph{the metric topology on $Z$}.  There is another natural topology on $Z$, which we will refer to as \emph{the quotient topology on $Z$}.  It is obtained as follows.  First, we equip $\til{Z}$ with the disjoint union topology, i.e., a set $A \subeq \til{Z}$ is open if and only for each $i \in I_0$, the set $A \cap X_i$ is open in $X_i$.  Note that a $\til{d}$-ball in $\til{Z}$ of finite radius is open. Then we consider the quotient topology on $Z$ arising from the equivalence relation $\sim$, i.e., the maximal topology on $Z$ in which the standard projection map $\pi \colon \til{Z} \to Z$ continuous.  It is not hard to check that every open set in the metric topology on $Z$ is open in the quotient topology on $Z$.  Moreover, if $\card I <\infty$, then the topologies coincide.  Simple examples show that the topologies may differ if this is not the case. 
\end{remark}

The following proposition states that away from the gluing sets, the space $Z$ is locally isometric to $\til{Z}$. 

\begin{proposition}\label{interior points estimate}  Let $r>0$, and suppose that $a \in Z$ satifies 
\begin{equation}\label{way inside}\dist_{d}\left(a, \pi\left(\bigcup_{i\in I}E_i\right)\right) \geq 3r.\end{equation}
Then there is unique representative $z_b$ of each $b \in B_d(a,r)$, and $\pi\inv\colon B_d(a,r) \to B_{\til{d}}(z_a,r)$ is a well-defined bijective isometry.
\end{proposition}

\begin{proof} Let $b$ and $c$ be points of $B_d(a,r)$.   Then Proposition \ref{comparison prop} and \eqref{way inside} imply that if $z_b$ and $z_c$ are representatives of $b$ and $c$ respectively, then 
\begin{equation}\label{way inside 2}\dist_{\til{d}}\left(\{z_b,z_c\}, \bigcup_{i\in I}(f_i(E_i)\cup E_i)\right) \geq 2r.\end{equation}
This immediately implies that $b$ and $c$ have unique representatives $z_b$ and $z_c$, respectively. Since $d(b,c) < 2r$, there is admissible sequence $\mbf{z}$ from $b$ to $c$ that has $\til{d}$-length less than $2r$.  The conditions \eqref{admissible} and \eqref{way inside 2} now imply that $\mbf{z} =z_b,z_c$, and hence that $d(b,c)=\til{d}(z_b,z_c)$.  A similar argument shows that $z_b$ and $z_c$ are in $B_{\til{d}}(z_a,r)$.  These facts together imply the desired statement.  
\end{proof}

\subsection{Preservation of the $\rm{ALLC}$ condition}\label{preserve ALLC section}
In this and the following subsection, we make the following assumption on the spaces in the collection $\{X_i\}_{i \in I}$.
\begin{itemize}
\item[(A)]\label{no bubbles} there is a constant $C\geq 1$ such that $\diam_{d_i}(X_i) \leq C\diam_{d_i}f_i(E_i)$ for all $i \in I$.
\end{itemize}
Heuristically, this assumption means that the spaces $(X_i,d_i)$ are ``flat".

 We now show that if each space in the collection $\{X_i\}_{i \in I_0}$ satisfies the $\rm{ALLC}$ condition with a uniform constant, then the glued space $(Z,d)$ also satisfies the $\rm{ALLC}$ condition, quantitatively.

\begin{lemma}\label{scoot} Suppose that there is a constant $\lambda \geq 1$ such that for each $i \in I_0$, the space $(X_i,d_i)$ is $\lambda$-$\rm{ALLC}$. Then there is a quantity $\Lambda \geq 1$, depending only on the data, with the following property.  Let $i \in I$, $a \in Z$, and $r>0$.  Then at least one of the following two statements holds: 
\begin{itemize}
\item[(i)] the annulus $A_d(a,r,2r)$ is contained in $\pi(X_i)$,
\item[(ii)] there is a point $b' \in A_d(a,r/\Lambda, 2\Lambda r) \cap \pi(E_i) $ such that for each point $b \in A_d(a,r,2r) \cap \pi(X_i)$, there is a continuum $E$ containing $b$ and $b'$ satisfying 
$$E \subeq A_d(a,r/\Lambda,2\Lambda r) \cap \pi(X_i).$$
\end{itemize}
\end{lemma}

\begin{proof}  We first assume that $a \notin \pi(X_i)$, and will show that the second statement above holds. We consider two sub-cases. 

\textit{Case 1: $r>32LC\lambda\dist_d(a,\pi(E_i))$.}  Let $a'$ be a point of $\pi(E_i)$ such that 
$$d(a,a') < \frac{r}{32LC\lambda}.$$  The triangle inequality implies that  
\begin{equation}\label{scoot 1} A_d(a,r,2r) \subeq A_d\left(a', \frac{r}{2},3r\right). \end{equation}
If $\diam_{d_i} f_i(E_i) < r/(2C)$, then Proposition \ref{comparison prop} and condition (A) show that $\pi(X_i) \subeq B_d\left(a',\frac{r}{2}\right).$
By \eqref{scoot 1}, this now implies that $A_d(a,r,2r) \cap \pi(X_i)$ is empty and hence the claim is vacuously true.  Thus we may assume that 
\begin{equation}\label{scoot 2} \diam_{d_i} f_i(E_i) \geq \frac{r}{2C}.\end{equation} 
Let $z_{a'}$ be the representative of $a'$ in $f_i(E_i)$. The connectedness of $f_i(E_i)$ and \eqref{scoot 2} imply that there is a point $z_{b'} \in f_i(E_i)$ such that 
$$d_i(z_{a'},z_{b'})=\frac{r}{8C}.$$
Set $b'=\pi(z_{b'}).$ Let $b \in A_d(a,r,2r)$, and denote by $z_{b}$ the representative of $b \in X_i$.  Proposition \ref{comparison prop}, the above equality, and \eqref{scoot 1} show that 
 $$z_b, z_{b'} \in A_{d_i}\left(z_{a'},\frac{r}{16C}, 3Lr\right).$$
 The $\lambda$-$\rm{ALLC}$ condition in $X_i$ now provides a continuum $E'$ containing $z_b$ and $z_b'$ such that 
 $$E' \subeq A_{d_i}\left(z_{a'},\frac{r}{16C\lambda}, 3L\lambda r\right).$$
 Proposition \ref{comparison prop} shows that 
 $$\pi(E') \subeq A_d\left(a', \frac{r}{16CL\lambda}, 3L\lambda r \right).$$ 
The triangle inequality now shows that 
$$
\pi(E')\subseteq A_d\left(a, \frac{r}{32CL\lambda}, \left(3L\lambda +1\right)r \right).
$$
Clearly $\pi(E')$ is a continuum containing $b$ and $b'$. 

\textit{Case 2: $r \leq 32LC\lambda \dist_d(a,\pi(E_i))$}. We may assume that there is a point $b_0 \in A_d(a,r,2r) \cap \pi(X_i)$, for otherwise statement (ii) above is vacuously true.  The definition of an admissible chain shows that this implies the existence of a point $b' \in \pi(E_i)$ satisfying $d(a,b') < 2r$ and $d(b,b') < 2r$. Given a point $b \in A_d(a,r,2r) \cap \pi(X_i)$, Proposition \ref{comparison prop} shows that we may find representatives $z_{b}$ and $z_{b'}$ in $X_i$ of $b$ and $b'$ respectively such that $d_i(z_b,z_{b'})<2Lr$.  By Lemma \ref{ALLC gives LLC}, the $\lambda$-$\rm{ALLC}$ condition in $X_i$ provides a continuum $E' \subeq B_{d_i}(z_b,4L\lambda r)$ that connects $b$ and $b'$.  Proposition \ref{comparison prop} implies that 
$$\pi(E') \subeq B_d(b,4L\lambda r) \subeq B_d(a,(4L\lambda + 2)r).$$
The restriction that $r \leq 32LC\lambda \dist_d(a,\pi(E_i))$ implies that $B_d(a,r/(32LC\lambda)$ does not intersect $\pi(E_i)$. The definition of an admissible sequence now shows that $B_d(a,r/(32LC\lambda)$ does not intersect $\pi(X_i)$.  Thus $\pi(E')$ is a continuum containing $b$ and $b'$ and satisfying $$\pi(E')\subseteq A_d\left(a,\frac{r}{32LC\lambda},(4L\lambda + 2)r\right).$$

Now, we assume that $a \in \pi(X_i)$, that the first statement above does not hold, and that the second statement above is not trivially true. That is, we assume that $A_d(a,r,2r)$ intersects both $\pi(X_i)$ and $Z \bslash \pi(X_i)$. 

We claim that $A_d(a,r/(4LC),2r)$ intersects $\pi(E_i)$. Since $\pi(E_i)$ is connected, if this is not the case, then either $\pi(E_i) \subeq \ovl{B}_d(a,r/(4LC))$ or $\pi(E_i) \subeq Z \bslash B_d(a,2r)$.  If the first possibility occurs, then condition (A) yields a contradiction with the assumptions that $A_d(a,r,2r)$ meets $\pi(X_i)$ and that $a \in \pi(X_i)$.  If the second possibility occurs, then the assumption that $A_d(a,r,2r)$ meets $Z \bslash \pi(X_i)$ and the definition of admissible chain yield a contradiction. 

Thus we may find a point $b' \in A_d(a,r/(4LC),2r) \cap \pi(E_i)$.  That this point satisfies the requirements of the second statement of lemma is left to the reader. 
\end{proof}

\begin{lemma}\label{X0 scoot} Suppose that there is a constant $\lambda \geq 1$ such that for each $i \in I_0$, the space $(X_i,d_i)$ is $\lambda$-$\rm{ALLC}$.  There is a quantity $\Lambda \geq 1$, depending only on the data, with the following property.  Let $a \in Z$ and $r>0$.  If $u$ and $v$ are points in $A_d(a,r,2r) \cap \pi(X_0)$, then there is a continuum $E$ containing $u$ and $v$ satisfying 
$$E \subeq A_d\left(a,\frac{r}{\Lambda},2\Lambda r\right).$$  
\end{lemma}

\begin{proof}  We claim that $\Lambda =4L\lambda +2$ fulfills the requirements of the lemma.  If $a \in \pi(X_0)$, this follows from Proposition \ref{comparison prop} and the $\lambda$-$\rm{ALLC}$ condition on $X_0$; the details are left to the reader.  Hence we assume that $a \notin \pi(X_0)$, and set 
$$s = \frac{L\lambda +1}{2}.$$
First, we consider the case that $r \leq s\dist_d(a,\pi(X_0))$.  Then $B_d(a,r/s)$ does not intersect $\pi(X_0)$, and so the existence of the desired continuum follows from Proposition \ref{comparison prop} and Lemma \ref{ALLC gives LLC}; again, the details are left to the reader. 

Now suppose that $r > s \dist_d(a, \pi(X_0))$.  Then there is a point $a' \in \pi(X_0)$ such that $d(a,a') =\dist_d(a,\pi(X_0)).$  The triangle inequality yields
$$A_d(a,r,2r) \subeq A_d\left(a',\frac{(s-1)r}{s}, \frac{(2s+1)r}{s}\right).$$
Let $z_u$, $z_v$, and $z_{a'}$ be representatives of $u$, $v$, $a'$, respectively, that are contained in $X_0$. Proposition \ref{comparison prop} and the $\lambda$-$\rm{ALLC}$ condition in $X_0$ imply that there is a continuum $E' \subeq X_0$ containing $z_u$ and $z_v$ satisfying 
$$ E' \subeq A_{d_0}\left(z_{a'},\frac{(s-1)r}{\lambda s}, \frac{(2s+1)L\lambda r}{s}\right).$$ 
Proposition \ref{comparison prop}, the triangle inequality, and the definition of $s$ now show that 
\begin{align*} \pi(E') &\subeq A_d\left(a', \frac{(s-1)r}{L\lambda s}, \frac{(2s+1)L\lambda r}{s}\right)\\ &\subeq A_d\left(a,\frac{r}{2L\lambda},(2L\lambda+2) r\right), \end{align*}
proving the claim in this case as well. 
\end{proof}

\begin{theorem}\label{LLC gluing} Suppose that there is a constant $\lambda \geq 1$ such that for each $i \in I_0$, the space $(X_i,d_i)$ is $\lambda$-$\rm{ALLC}$.  Then $(Z,d)$ is $\Lambda$-$\rm{ALLC}$, where $\Lambda \geq 1$ depends only on the data.
\end{theorem}

\begin{proof}  Let $a \in Z$ and $r>0$.  We show that each pair of points $u,v \in A_d(a,r,2r)$ is contained in continuum $E \subeq Z$ satifying
$$ E \subeq A_d\left(a, \frac{r}{\Lambda}, 2\Lambda r\right),$$
where $\Lambda \geq 1$ is now defined to be the maximum of the quantities provided by Lemmas \ref{scoot} and \ref{X0 scoot}. 

Choose representatives $z_u \in X_{i_u}$ and $z_v \in X_{i_v}$ of $u$ and $v$, respectively. If $i_u=i_v \in I$, then Lemma \ref{scoot} provides the desired continuum. If $i_u=i_v=0$, then Lemma \ref{X0 scoot} provides the desired continuum.  If $i_u \neq i_v$ and neither are $0$, we employ Lemma \ref{scoot}. If the first possibility in Lemma \ref{scoot} holds, the desired continuum is easily constructed using Proposition \ref{comparison prop}. If the second possibility holds, we are provided with continua $E_u$ and $E_v$ that connect $u$ and $v$ to points $u' \in \pi(X_0)$ and $v' \in \pi(X_0)$ respectively, where 
$$E_u \cup E_v \subeq A_d(a,r/\Lambda,2\Lambda r).$$
If $i_u = 0$, then we instead set $u'=u$, and similarly define $v'=v$ if $i_v=0$. Applying Lemma \ref{X0 scoot} to $u'$ and $v'$ and concatenating now produces the desired continuum. 
\end{proof}

\subsection{Preservation of Ahlfors regularity}\label{preserve reg section}
In this subsection only, we add to condition (A) two assumptions on the geometry of the base space $X_0$. First, we assume the uniform relative separation of the gluing sets:
\begin{itemize}
\item[(B)] \label{rel sep} there is a constant $c>0$ such that $\bigtriangleup(E_i,E_j)\geq c$ for all $i \neq j \in I$.
\end{itemize}
We also assume, without loss of generality, that $c \leq 1$. 

Our final assumption corresponds to condtion \eqref{planarity} in the statement of Theorem~ \ref{finite rank}. Namely, for each integer $k \geq 0$, set 
$$n_k := \sup \card \left\{i \in I:  E_i \cap B_{d_0}(z,r)\neq \emptyset \ \text{and}\ 2^{-k} < \frac{\diam_{d_0}(E_i)}{r} \leq 2^{-k+1}\right\},$$
where the supremum is taken over all $z \in X_0$ and $0<r \leq 2\diam_{d_0}X_0$.  We assume that there are numbers $Q>0$ and $1 \leq M \leq \infty$ such that
\begin{itemize}
\item[(C)]\label{not too many holes} $\sum_{k \in \nats} n_k (2^{-k})^Q \leq M.$
\end{itemize}

\begin{theorem}\label{regular gluing} Suppose that there is a constant $K \geq 1$ such that for each $i \in I_0$, the space $(X_i,d_i)$ is Ahlfors $Q$-regular with constant $K$. Then $(Z,d)$ is Ahlfors $Q$-regular, quantitatively.  \end{theorem}

\begin{proof} As noted in the definition of Ahlfors regularity, it suffices to show that for $a \in Z$ and $0<r \leq 2\diam_d Z,$ 
\begin{equation}\label{reg preserve goal}\Hdim^Q_{(Z,d)}(B_d(a,r)) \simeq r^Q.\end{equation}

First suppose that $a$ satisfies 
$$\dist_{d}\left(a, \pi\left(\bigcup_{i\in I}E_i\right)\right) \geq 3r.$$
Then Proposition \ref{interior points estimate} implies that 
$$\Hdim^Q_{(Z,d)}(B_d(a,r)) = \Hdim^Q_{(\til{Z},\til{d})}(B_{\til{d}}(z_a,r)) =\Hdim^Q_{(X_i,d_i)}(B_{d_i}(z_a,r)),$$
where $z_a \in X_i$ is the unique representative of $a$.  It also follows that $r$ can be no larger than twice the diameter of $(X_i,d_i)$. Since $(X_i,d_i)$ is Ahlfors $Q$-regular, the desired estimate follows.

Next we suppose that $a \in \pi(E_i)$ for some $i \in I$. Let $z^0_a \in E_i$ and $z_a=f_i(z^0_a) \in f_i(E_i) \subeq X_i$ be the representatives of $a$.  Set $B_0 = B_d(a,r) \cap \pi(X_0)$. 
Proposition \ref{comparison prop} implies that $\pi|_{X_0}$ is an $L$-bi-Lipschitz mapping, and that 
$$\pi\left(B_{\til{d}}(z^0_a,r)\right) \subeq B_0 \subeq \pi\left(B_{\til{d}}(z^0_a,Lr)\right).$$ 
Note that by the triangle inequality, Proposition \ref{comparison prop}, and condition (A),
$$ \diam_d Z  \simeq \diam_{d_0} X_0.$$
Hence, we may apply the Ahlfors $Q$-regularity of $(X_0,d_0)$ to see that
\begin{equation}\label{X0 estimate}\Hdim^Q_{(Z,d)}(B_0) \simeq \Hdim^Q_{(X_0,d_0)}(B_{d_0}(z^0_a,r)) \simeq r^Q.\end{equation}
Thus 
$$\Hdim^Q_{(Z,d)}(B_d(a,r)) \gtrsim r^Q.$$ 

We now work towards an upper bound for $\Hdim^Q_{(Z,d)}(B_d(a,r))$. 
For $j \in I$ set $B_j = B_d(a,r) \cap \pi(X_j),$ and let $J \subeq I$ be the set of indices such that $B_j \neq \emptyset$. Furthermore, for $k \in \ints$, define
$$J^k = \left\{j \in J : 2^{-k} < \frac{\diam_{d_0}(E_j)}{Lc\inv r} \leq 2^{-k+1}\right\}.$$
Now, we may write
$$B_d(a,r) = B_0 \cup \bigcup_{k \in \ints}\bigcup_{j \in J^k} B_j.$$

By definition, $B_j \subeq \pi(X_j)$ for any $j \in J$.  As before, by Proposition \ref{comparison prop}, $\pi|_{X_j}$ is an $L$-bi-Lipschitz mapping. Thus the Ahlfors $Q$-regularity of each $X_j$ and condition (A) imply that 
\begin{equation}\label{typical}\Hdim^Q_{(Z,d)}(B_j) \leq \Hdim^Q_{(Z,d)}(\pi(X_j)) \simeq \Hdim^Q_{(X_j,d_j)}(X_j)  \simeq (\diam_{d_0} E_j)^Q.\end{equation}

Fix $j \in J$. We claim that $E_j \cap B_{d_0}(z^0_a,Lr) \neq\emptyset$. By definition, we may find a point $b \in \pi(X_j)$ and an admissible chain $\mbf{z}$ from $a$ to $b$ of $\til{d}$-length less than $r$.  It follows from the definitions that there is an admissible sub-chain $\mbf{z}'$ connecting $a$ to a point $b' \in \pi(E_j)$, and the $\til{d}$-length of $\mbf{z}'$ is also less than $r$.  Let $z^0_{b'}$ be the representative of $b'$ in $E_j$.  By Proposition \ref{comparison prop}, it holds that
$$d_0(z^0_a,z^0_{b'}) < Lr,$$
proving the claim. Since we have assumed that $c \leq 1$, the claim implies that $\card J^k \leq n_k$ for each integer $k \geq 0$.

It now follows from condition (B) that there is at most one index $j \in J$ with the property that $\diam_{d_0} E_j > 2Lc\inv r$, and hence 
$$\card\left(\bigcup_{k \leq -1} J^k\right) \leq 1.$$
Suppose $j_0 \in J$ is an index with the above property, and let $a_{j_0}$ be a point of $B_{j_0}$. Then by the triangle inequality, 
$$B_{j_0} \subeq B_d(a_{j_0},2r) \cap \pi(X_{j_0}).$$
Proposition \ref{comparison prop} implies that $\pi|_{X_{j_0}}$ is an $L$-bi-Lipschitz mapping onto $\pi(X_{j_0})$.  Hence the Ahlfors $Q$-regularity of $X_{j_0}$ implies that 
\begin{equation}\label{big one}\Hdim^Q_{(Z,d)}(B_{j_0}) \leq \Hdim^Q_{(Z,d)}\left(B_{d}(a_{j_0},2r) \cap \pi(X_{j_0})\right) \lesssim r^Q.\end{equation}

Thus, inequalities \eqref{X0 estimate}, \eqref{typical}, and \eqref{big one}, along with condition (C), imply that 
\begin{align*} \Hdim^Q_{(Z,d)}(B_d(a,r)) &\leq \Hdim^Q_{(Z,d)}(B_0) + \sum_{k \in \ints}\sum_{j \in J^k}\Hdim^Q(B_j) \\ &\lesssim r^Q + \sum_{k=0}^\infty \sum_{j \in J^k}\diam_{d_0}(E_j)^Q \\ &\lesssim r^Q +  r^Q\left(\sum_{k=0}^\infty \card(J^k)(2^{-k})^Q\right) \lesssim r^Q,\end{align*}
as desired. Note that this upper bound is also valid when $r>2\diam_d Z$. 

Finally, we consider the full case that 
$$\dist_{d}\left(a, \pi\left(\bigcup_{i\in I}E_i\right)\right) < 3r.$$
We may find an index $i_0 \in I$ and a point $b \in \pi(E_{i_0})$ such that $d(a,b) < 3r.$
Thus the triangle inequality and the previous case show that 
$$\Hdim^Q_{(Z,d)}(B_d(a,r)) \leq  \Hdim^Q_{(Z,d)}(B_d(b,4r)) \lesssim r^Q.$$ 
To get the desired lower bound, we consider two subcases. If 
\begin{equation}\label{r half}\dist_{d}\left(a, \pi\left(\bigcup_{i\in I}E_i\right)\right) < r/2,\end{equation}
then, as above, we may find an index $i_0 \in I$ and a point $b \in \pi(E_{i_0})$ such that 
$d(a,b)<r/2$. The triangle inequality and the previous case now show that
$$\Hdim^Q_{(Z,d)}(B_d(a,r))  \geq \Hdim^Q_{(Z,d)}(B_d(b,r/2)) \gtrsim r^Q.$$ 
If \eqref{r half} does not hold, then setting $r'=r/6$, we see that
$$\dist_{d}\left(a, \pi\left(\bigcup_{i\in I}E_i\right)\right) \geq 3r',$$
and we may apply the first case considered in the proof to conclude that 
$$\Hdim^Q_{(Z,d)}(B_d(a,r))  \geq \Hdim^Q_{(Z,d)}(B_d(a,r')) \gtrsim (r')^Q \simeq r^Q,$$
as desired. \end{proof}

\section{putting it together}\label{putting it together}
In this section, we synthesize the results of the previous sections to produce a proof of our main result.  We begin by setting up an induction.

Let $Y$ be a metric space. Given a subset $\mathcal{S}$ of $\mathcal{C}(Y)$, denote by $\cl{\mathcal{S}}$ the topological closure of $\mathcal{S}$ in $\mathcal{C}(Y)$. Let $\mathcal{N}(Y) \subeq \mathcal{C}(Y)$ denote the collection of non-trivial components of $\partial{X}$, and let $\mathcal{I}(Y)$ denote the points of $\mathcal{N}(Y)$ that are isolated points of $\mathcal{C}(Y)$. 

\begin{lemma}\label{together top}Let $(X,d)$ be a metric space, homeomorphic to a domain in $\sphere$, such that conditions \eqref{2 reg condition}-\eqref{rel sep condition} of Theorem \ref{finite rank} hold. Then $(X,d)$ bi-Lipschitzly embeds into a metric space $(Z,d_Z)$ that is homeomorphic to a domain in $\sphere$, satisfies conditions \eqref{2 reg condition}-\eqref{rel sep condition} of Theorem \ref{finite rank} quantitatively, and such that $\mathcal{C}(Z)$ is homeomorphic to $\mathcal{C}(X)\bslash \mathcal{I}(X)$. 
\end{lemma}

\begin{proof} We leave the verification of the quantitativeness of the statement to the reader, as it follows easily from the quantitativeness of the results proven thus far.

Denote $\mathcal{I}(X) =\{E_i\}_{i \in I}$. Since $(X,d)$ has compact completion, the index set $I$ has cardinality no larger than countably infinite. 

Fix $i \in I$.  Since $X$ is Ahlfors $2$-regular, it is doubling. Lemma \ref{ALLC gives LLC} implies that $X$ is $\rm{LLC}$. Hence, Theorem \ref{boundary uniformization} and Corollary \ref{bdry porosity ALLC} imply that $E_i$ is a quasicircle that is porous in $\ovl{X}$.  Remark \ref{Herron Meyer remark} provides an Ahlfors $2$-regular quasidisk $D_i \subeq \sphere$ with the property that there is a bi-Lipschitz map $f_i \colon E_i \to \partial D_i$. It is easily seen by using \cite[Proposition 10.10]{LAMS} that the $\rm{ALLC}$ property is preserved by quasisymmetric mappings. Hence, $D_i$ is $\rm{ALLC}$. Note that none of the data of the conditions discussed in this paragraph depend on $i$. 

We now apply the results of Section \ref{gluing section}. Let $X_0$ be the completion $\ovl{X}$, and for each $i \in I$ set $X_i = \ovl{D_i}.$ We employ the bi-Lipschitz maps $f_i$ defined above as the gluing maps $f_i \colon E_i \to f_i(E_i) \subeq X_i.$  The conditions stated at the beginning of Section \ref{gluing section} are met by construction, and hence we may consider the resulting glued metric space $(Z,d_Z)$. 

We first show that $Z$ is homeomorphic to a domain in $\sphere$. The proof is similar in spirit to that of Theorem \ref{isolate closure uniformization}. Denote 
$$\til{X} = X \cup \left( \bigcup_{i \in I} E_i \right).$$
By Theorem \ref{isolate closure uniformization}, there is a circle domain $\Omega' \subeq \sphere$ and a homeomorphism 
$$\til{h} \colon \til{X} \to \Omega' \cup \left(\bigcup_{F \in \mathcal{N}(\Omega')} F \right).$$
Denote the image of $\til{h}$ by $\til{\Omega'}.$ Since $\Omega'$ is a circle domain, for each $i \in I$, we may write $\til{h}(E_i) = S_\sphere(p_i,r_i),$ where $p_i \in \sphere$ and $r_i>0$. Since $\til{h}$ induces a homeomorphism from $\mathcal{C}(X)$ to $\mathcal{C}(\til{\Omega'})$, for each $i \in I$ there is a number $\ep_i > 0$ such that 
$$\ovl{B}_\sphere(p_i,r_i+\ep_i) \cap \partial{\Omega'} = S_\sphere(p_i,r_i)$$ and such that the resulting collection $\{\ovl{B}_\sphere(p_i,r_i+\ep_i)\}_{i \in I}$ is pairwise disjoint. Moreover, the set  
$$\Psi = \Omega' \cup \left(\bigcup_{i \in I} \ovl{B}_{\sphere}(p_i,r_i)\right)$$
is a domain in $\sphere$.

Proposition \ref{comparison prop} implies that for each $i \in I_0$, there is a bi-Lipschitz homeomorphism $\iota_i \colon X_i \to \pi(X_i)$.  Fix $i \in I$.  There is a homeomorphism $g_i \colon X_i \to \ovl{B}_\sphere(p_i,r_i)$.  By Lemma \ref{untwist}, there is a homeomorphism $\Phi_i$ of  $\ovl{A}_\sphere(p_i,r_i,r_i+\ep_i)$ to itself such that $\Phi_i$ coincides with the identity on $S_\sphere(p_i,r_i+\ep_i)$ and coincides with $g_i \circ \iota_i\inv \circ \iota_0 \circ \til{h}\inv$ on $S_\sphere(p_i,r_i)$.  
Now, the map $H \colon Z \to \Psi$ defined by 
$$H(a) = \begin{cases}
				\til{h} \circ \iota_0\inv(a) & a \in \pi(X_0)\bslash \left(\bigcup_{i \in I} \iota_0\circ \til{h}\inv(\ovl{B}_\sphere(p_i,r_i+\ep_i))\right), \\
				 \Phi_i\circ \til{h} \circ \iota_0\inv(a) & a \in \iota_0 \circ \til{h}\inv(\ovl{A}_\sphere(p_i,r_i, r_i+\ep_i)),\\
				 g_i \circ \iota_i \inv(a) & a \in \pi(X_i), \end{cases}$$
is the desired homeomorphism. See Figure \ref{glue}.

\begin{figure}[h]
\begin{center}
\psfrag{pi}{$p_i$}
\psfrag{ei}{$\ep_i+r_i$}
\psfrag{ri}{$r_i$}
\psfrag{ph}{$\Phi_i$}
\psfrag{X0}{$X_0$}
\psfrag{Xi}{$X_i$}
\psfrag{pX0}{$\pi(X_0)$}
\psfrag{pXi}{$\pi(X_i)$}
\psfrag{i0}{$\iota_0$}
\psfrag{ii}{$\iota_i$}
\psfrag{gi}{$g_i$}
\psfrag{Ei}{$E_i$}
\psfrag{th}{$\til{h}$}
\includegraphics[width=.75\textwidth]{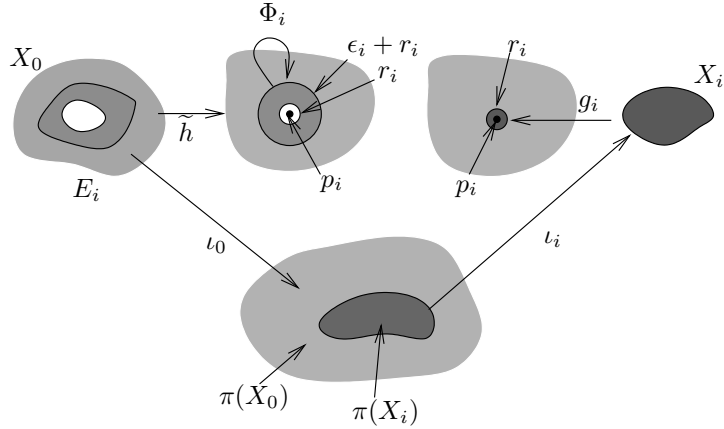}
\caption{Filling in holes.}\label{glue}
\end{center}
\end{figure}

We now verify that $Z$ satisfies conditions \eqref{2 reg condition}-\eqref{rel sep condition} of Theorem \ref{finite rank}. Condition (A), which was imposed at the beginning of Subsection \ref{preserve ALLC section}, is equivalent to the assertion that $\diam D_i \lesssim \diam \partial{D_i}$, which follows from the fact that $D_i$ is a quasidisk or from the fact that $D_i$ is planar. Conditions (B) and (C), which were imposed at the beginning of Subsection \ref{preserve reg section}, follow from the assumptions \eqref{rel sep condition} and \eqref{planarity} in the statement of Theorem \ref{finite rank}. 
 
Hence, Theorems \ref{LLC gluing} and \ref{regular gluing}  imply that $(Z,d_Z)$ is $\rm{ALLC}$ and Ahlfors $2$-regular.  That $Z$ satisfies the remaining conditions \eqref{planarity}, \eqref{compact condition}, \eqref{rel sep condition} of Theorem \ref{finite rank} and that $\mathcal{C}(Z)$ is homeomorphic to $\mathcal{C}(X)\bslash \mathcal{I}(X)$ follow from the construction and Proposition \ref{comparison prop}; we leave the details to the reader.
\end{proof}

\begin{proof}[Proof of Theorem \ref{finite rank}] The necessity of conditions \eqref{compact condition}-\eqref{rel sep condition} follows easily from the basic properties of quasisymmetric mappings and Proposition \ref{ALLC rel sep circ}.

Now, let $(X,d)$ be a metric space, homeomorphic to a domain in $\sphere$, such that the closure of the non-isolated components of $\mathcal{C}(X)$ is countable and has finite rank, and such that conditions \eqref{2 reg condition}-\eqref{rel sep condition} hold. We will show that $(X,d)$ is quasisymmetrically equivalent to a circle domain whose collection of boundary components are uniformly relatively separated. Again, we leave the issue of quantitativeness to the reader.

We first reduce to the case that $\cl\mathcal{N}(X) = \mathcal{C}(X)$.   Let $\mathcal{T}=\mathcal{C}(X)\bslash \cl\mathcal{N}(X)$.  Consider the subspace $\til{X}$ of $\ovl{X}$ defined by 
$$\til{X} = X \cup \left(\bigcup_{E \in \mathcal{T}} E\right).$$
Let $h \colon \ovl{X} \to \sphere$ be the continuous surjection provided by Corollary \ref{good map}.  Since each $E \in \mathcal{T}$ is trivial, Corollary \ref{good map} implies that the map $h|_{\til{X}}$ is a homeomorphism. Moreover, $h$ induces a homeomorphism of $\mathcal{C}(X)$ onto the totally disconnected set $\sphere\bslash h(X)$. Hence, the set $\{h(E)\}_{E \in \mathcal{T}}$ is open in $\sphere\bslash h(X)$.  It follows that the image $h(\til{X})$ is a domain in $\sphere$.  Remark \ref{extend reg to boundary} and Propostion \ref{BetterALLC} imply that $\til{X}$ is Ahlfors $2$-regular and $\rm{ALLC}$, quantitatively. Moreover, the space $\til{X}$ clearly satisfies the remaining assumptions of Theorem \ref{finite rank}, since $\mathcal{C}(\til{X}) = \mathcal{C}(X) \bslash \mathcal{T}$ and hence 
$\cl\mathcal{N}(\til{X}) = \cl\mathcal{N}(X) = \mathcal{C}(\til{X}).$

Thus, if Theorem \ref{finite rank} is valid for spaces such that the non-trivial boundary components are dense in the space of all boundary components, then applying the theorem to $\til{X}$ and restricting the resulting quasisymmetric mapping to $X$ proves the theorem for $X$.

We now assume without loss of generality that $\cl \mathcal{N}(X) = \mathcal{C}(X)$. Our assumptions now imply that $\mathcal{C}(X)$ is countable and has finite rank, and we proceed by induction on the rank. If the rank of $\mathcal{C}(X)$ is $0$, then every boundary component is isolated. By Lemma \ref{together top}, there is a bi-Lipschitz embedding $\iota \colon X \into Z$ where $(Z,d_Z)$ is complete, homeomorphic to a domain in $\sphere$, and satisfies conditions \eqref{2 reg condition}-\eqref{rel sep condition}. Condition \eqref{compact condition} implies that $Z$ is compact, and hence homeomorphic to $\sphere$.  Lemma \ref{ALLC gives LLC} implies that $Z$ is $\rm{LLC}$. Bonk and Kleiner's uniformization result, Theorem \ref{two sphere}, now provides a quasisymmetric homeomoprhism $f \colon Z \to \sphere$.  By Proposition \ref{comparison prop}, there is a bi-Lipschitz embedding $\iota \colon X \into Z$.  The composition $f \circ\iota$ is a quasisymmetric homeomorphism onto its image $\Omega :=f \circ \iota(X)$.  This mapping extends to a quasisymmetric homeomorphism of $\ovl{X}$ onto $\ovl{\Omega}$ \cite[Proposition 10.10]{LAMS}. As discussed in the proof of Lemma \ref{together top}, each of the components $\Gamma_1,\hdots,\Gamma_N$ of $\partial{X}$ is a quasicircle. Hence $\partial{\Omega}$ consists of finitely many quasicircles $\{f\circ \iota(\Gamma_1),\hdots,f\circ \iota(\Gamma_N)\}$, and by Remark \ref{rel sep preservation}
$$\min_{i\neq j \in \{1,\hdots,N\}} \bigtriangleup(f\circ \iota(\Gamma_i),f\circ \iota(\Gamma_j)) \simeq \min_{i\neq j \in \{1,\hdots,N\}} \bigtriangleup(\Gamma_i, \Gamma_j).$$
Thus Bonk's uniformization result in $\sphere$, Theorem \ref{two sphere Koebe}, provides a quasisymmetric homeomorphism $g \colon \sphere \to \sphere$ with the property that $g \circ f \circ \iota(X)$ is a circle domain. Again, the relative separation of the boundary components of this domain is controlled by Remark \ref{rel sep preservation}. This completes the proof in the case that the rank of $\mathcal{C}(X)$ is $0$. 

We now assume that the desired statement is true in the case that the rank of $\mathcal{C}(X)$ is an integer $k \geq 1$, and suppose that the rank of $\mathcal{C}(X)$ is $k+1$.   Theorem \ref{together top} now states that $X$ bi-Lipschitzly embedds into a metric space $(Z,d_Z)$ that is homeomorphic to a domain in the sphere, satisfies conditions \eqref{2 reg condition}-\eqref{rel sep condition}, and such that $\mathcal{C}(Z)$ has rank $k$. By induction, $Z$ is quasisymmetrically equivalent to a circle domain. The remainder of the proof proceeds as in the base case. \end{proof}

\begin{proof}[Proof of Theorem \ref{finite case}] This follows from Proposition \ref{LLC gives ALLC} and Theorem \ref{finite rank}, after noting that the minimal relative separation of components of the boundary is controlled by the ratio of the minimal distance between components of the boundary to the diameter of the space. 
\end{proof}

\begin{remark}\label{non compact} We have defined a circle domain to be a subset of $\sphere$; one may also consider circle domains in $\reals^2$, which need not have compact completion.  An analogous version of Theorem \ref{finite rank} for such domains can easily be derived from Theorem \ref{finite rank} and the techniques of \cite[Section 6]{QSPlanes}. 
\end{remark}

\section{The counter-example}\label{counter-example section}

In this section, we prove Theorem \ref{example}. 

The desired space $(X,d)$ is obtained as follows. First we define a sequence  of multiply connected domains $(Q_n)$.
Let $Q_0$ denote the open unit square $(0,1)\times(0,1)$ in the plane. Let $Q_1$ be the domain obtained by removing the vertical line segment $\{1/2\}\times[1/4,3/4]$ from $Q_0$. We define $Q_{n+1}$ by subdividing $Q_{n}$ into $2^n\times2^{n}$ dyadic subsquares of equal size in the obvious way, and replacing each square in the subdivision by a copy of $Q_1$ that has been scaled by $1/2^n$.
See~Figure~\ref{F:Seq} for $Q_1, Q_2$, and $Q_3$. We denote by $\ovl Q_n$ the completion of $Q_n$ in the path metric $d_{Q_n}$ induced from the plane. 
The induced metric on the completion is denoted by $d_{\ovl Q_n}$. 
The boundary components of $Q_n$ give rise to the metric boundary components of $(Q_n,d_{Q_n})$. The components of the metric boundary of  $(Q_n,d_{Q_n})$ that correspond to the boundary components of $Q_n$ other than the outer boundary will be called the \emph{slits} of $\ovl Q_n$.

\begin{figure}
[htbp]
\begin{center}
\includegraphics[width=.3\textwidth]{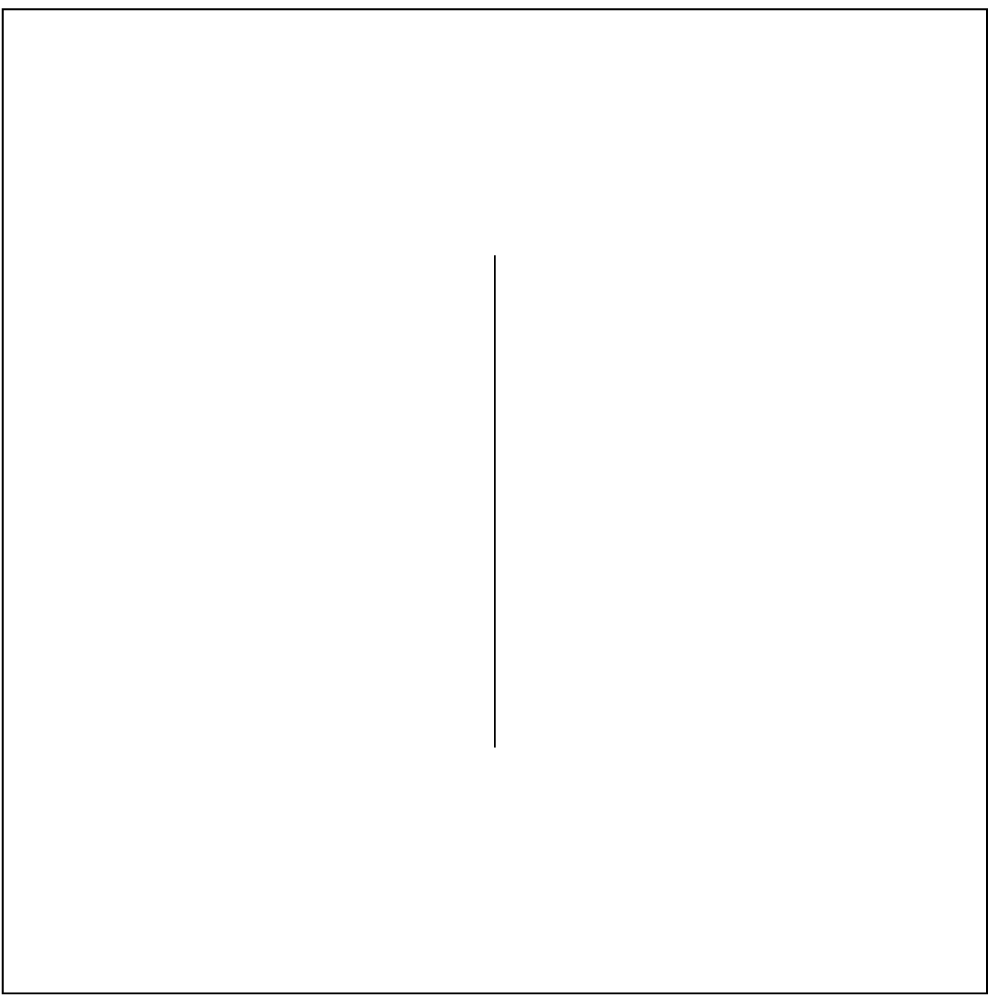}
\hspace{10pt}
\includegraphics[width=.3\textwidth]{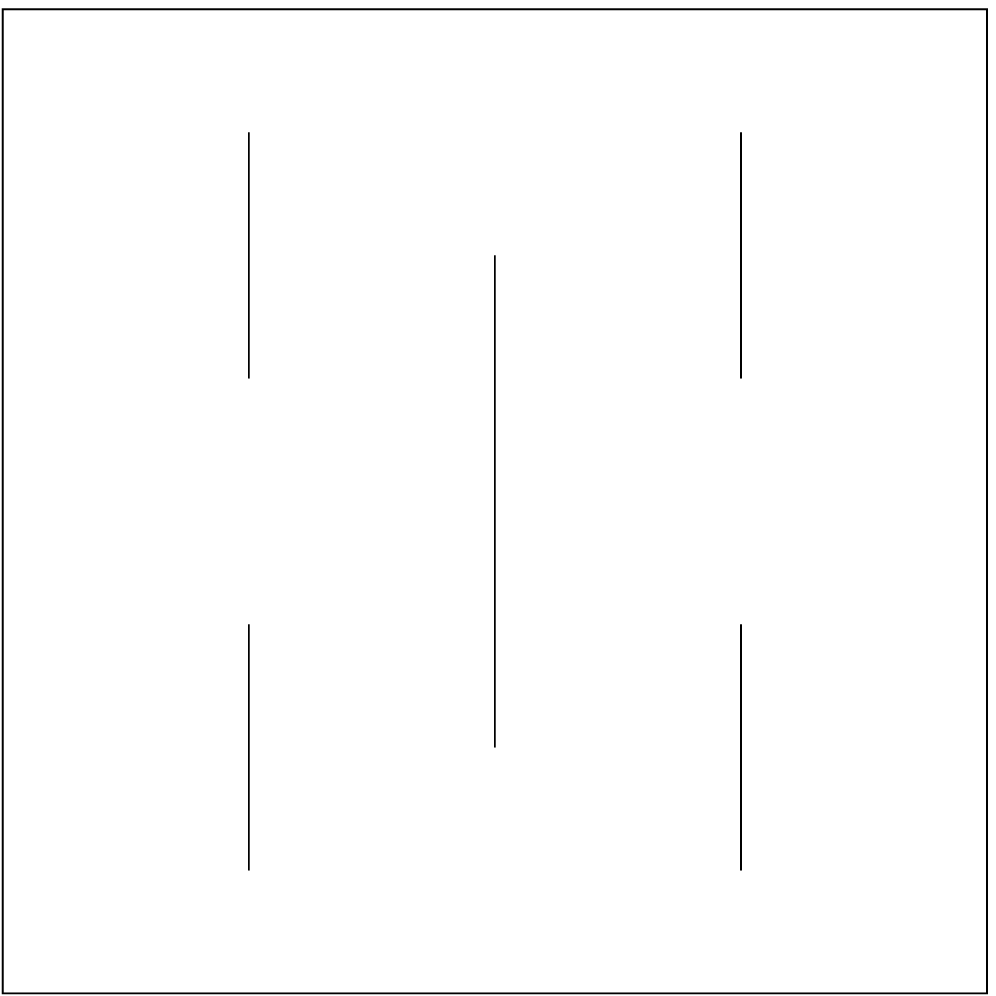}
\hspace{10pt}
\includegraphics[width=.3\textwidth]{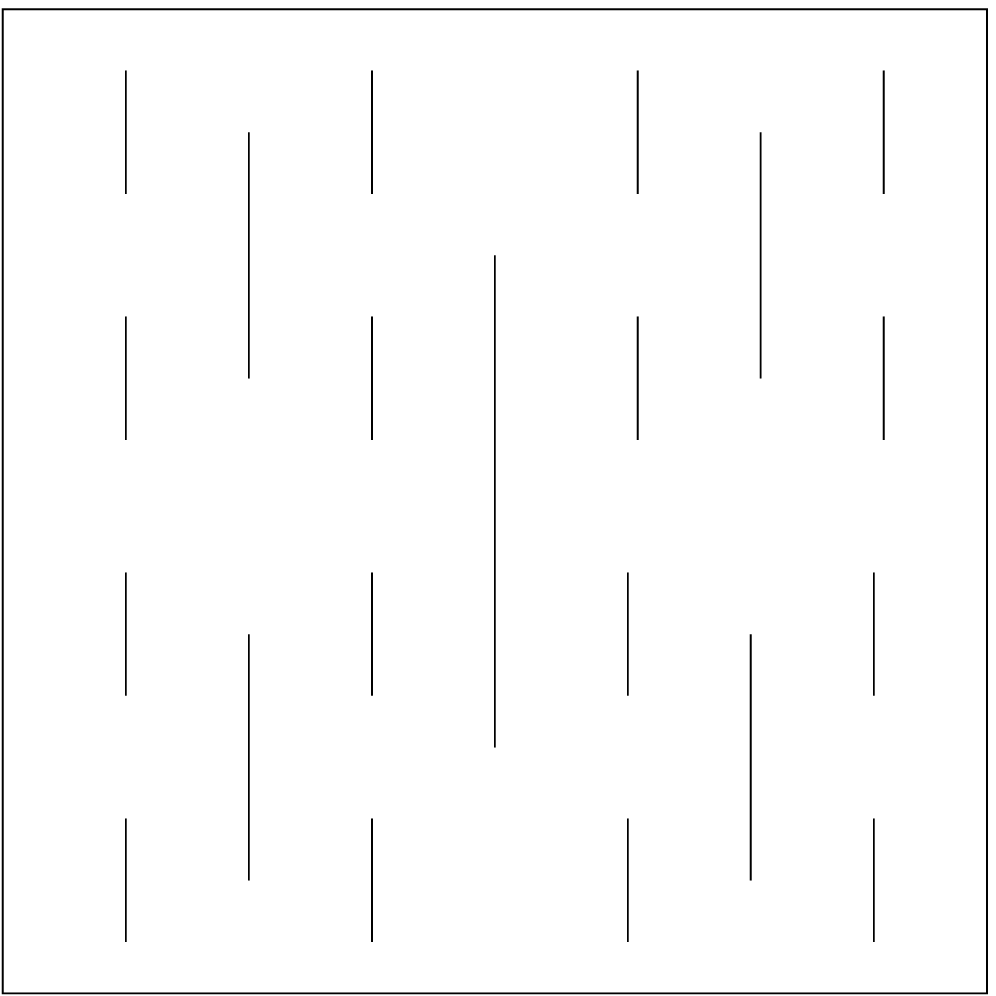}
\caption{Domains $Q_1, Q_2$ and $Q_3$.}
\label{F:Seq}
\end{center}
\end{figure}

Next we define a planar domain $Q$ inductively as follows. We start with $Q_1$ and replace the left-lower subsquare $(0,1/2)\times(0,1/2)$ by a copy of $Q_2$ that has been scaled by a factor of $1/2$. The resulting domain is denoted by $R_1$. Then we replace $R_1\cap((0,1/4)\times(0,1/4))$ by a copy of $Q_3$ that has been scaled by a factor of $1/4$. The resulting domain is denoted by $R_2$. We continue in this fashion and at $n$th step we replace $R_{n-1}\cap((0,1/2^n)\times(0,1/2^n))$ by a copy of $Q_{n+1}$ that has been scaled by a factor of $1/2^n$. The resulting domain is denoted by $R_n$. The countably connected domain that  results after infinitely many such replacements is denoted by $Q$, see Figure~\ref{F:SpaceZ}.

\begin{figure}
[htbp]
\begin{center}
\includegraphics[height=40mm]{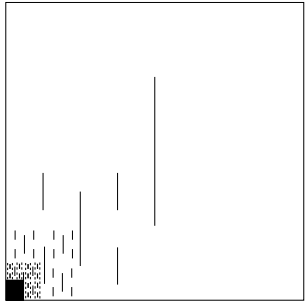}
\caption{Domain $Q$.}
\label{F:SpaceZ}
\end{center}
\end{figure}

The desired metric space $(X,d)$ is the domain $Q$ endowed with the path metric $d_{Q}$ induced from the plane.  We denote by $\pi_{\ovl{X}}$ the projection of $\ovl{X}$ onto $\ovl Q_0$, i.e., the map that identifies the points in $\ovl{X}$ that correspond to the same point of $\ovl Q_0$. The map $\pi_{\ovl{X}}$  is clearly 1-Lipschitz since the path metric on $Q$ dominates the Euclidean metric. 

Alternatively, the space $\ovl{X}$ can be defined as an inverse limit of the sequence of metric spaces $(\ovl R_n,p_{mn}),\ m\leq n$. Here $\ovl R_n$ is the completion of $R_n$ in the path metric induced from the plane and $p_{mn}$ is the projection of $\ovl R_n$ onto $\ovl R_m$ that identifies the points of $\ovl R_n$ that correspond to the same point of $\ovl R_m$. The map $\pi_{\ovl{X}}$ is then the natural projection of $\ovl{X}$ onto $\ovl Q_0$. The extended metric $d$ on $\ovl{X}$ now satisfies the equation
$$d(p,q)=\lim d_{\ovl R_n}(p_n,q_n),
$$
where $\{p_n \in \ovl{R_n}\}$ and $\{q_n \in \ovl{R_n}\}$ are sequences corresponding to $p$ and $q$ in the inverse limit system. 

To establish the desired properties of the space $X$, we consider the slit carpet $S_2$ that has been studied in~\cite{sM09}. The space $S_2$ is the inverse limit of the system $(\ovl Q_n, \pi_{mn}),\ m\leq n$, where $\pi_{mn}$ is the projection of $\ovl Q_n$ onto $\ovl Q_m$ that identifies the points on the slits of $\ovl Q_n$  that correspond to the same point of $\ovl Q_m$.
See Figure~\ref{F:Slitc}.

\begin{figure}
[htbp]
\begin{center}
\includegraphics[height=40mm]{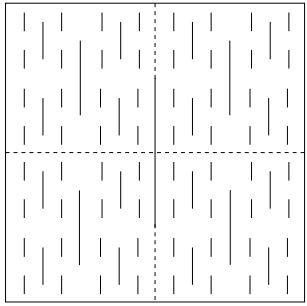}
\caption{Slit carpet $S_2$.}
\label{F:Slitc}
\end{center}
\end{figure}

We make $S_2$ into a metric space by endowing it with the metric 
$$
d_{S_2}(p,q)=\lim d_{\ovl Q_n}(p_n,q_n),
$$
where $\{p_n \in \ovl{Q_n}\}$ and $\{q_n \in \ovl{Q_n}\}$ are sequences corresponding to $p$ and $q$ in the inverse limit system. The space $S_2$ is a geodesic metric space which is a \emph{metric Sierpi\'nski carpet}, i.e., a metric space homeomorphic to the well known standard Sierpi\'nski carpet, see~\cite[Lemma~2.1]{sM09}. The inverse limits of slits of $\ovl Q_n$ form the family of \emph{peripheral circles} of $S_2$, i.e., embedded simple closed curves whose removal does not separate $S_2$. The natural projection $\pi_{S_2}$ of $S_2$ onto $\ovl Q_0$ factors through the projection $\pi_{S_2\to \ovl{X}}$ of $S_2$ onto $\ovl{X}$, i.e., 
$$
\pi_{S_2}=\pi_{\overline X}\circ\pi_{S_2\to\overline X}.
$$
The projection $\pi_{S_2\to \ovl{X}}$ is clearly 1-Lipschitz.

\begin{lemma}\label{L:Incl}
 There exists a constant $c>0$ such that for every $p\in \ovl{X}$ and every $0\leq r<2{\rm diam}(X)$, there exists $q\in\ovl Q_0$ with
$$
B_{\ovl Q_0}(q,c\cdot r)\subseteq \pi_{\ovl{X}}(B_{\ovl{X}}(p,r))\subseteq B_{\ovl Q_0}(\pi_{\ovl{X}}(p),r).
$$
\end{lemma}

\begin{proof}
Fix $p\in\ovl{X}$ and $0<r<2{\rm diam}(X)$. The  second inclusion follows since $\pi_{\ovl{X}}$ is 1-Lipschitz. 
To show the first inclusion, we use the corresponding property of $S_2$ proved in~\cite{sM09}. Since $\pi_{S_2\to\ovl{X}}$ is 1-Lipschitz, for any $p'\in \pi_{S_2\to\ovl{X}}^{-1}(p)$ we have
$$
\pi_{S_2\to\ovl{X}}(B_{S_2}(p',r))\subseteq B_{\ovl{X}}(p,r).
$$  
By~\cite[Lemma~2.2]{sM09}, there exists $c>0$ such that for any $p'$ as above, there is $q\in\ovl Q_0$ with 
$$
B_{\ovl Q_0}(q,c\cdot r)\subseteq\pi_{S_2}(B_{S_2}(p',r)).
$$
Combining these inclusions with the factorization of $\pi_{S_2}$ yields the desired inclusion.
\end{proof}

The following lemma implies that the Lipschitz map $\pi_{\ovl{X}}$ is David--Semmes regular, see~\cite[Definition~12.1]{Fractured}.
\begin{lemma}\label{L:Reg}
 There exists $C\geq 1$ such that for every $q\in \ovl Q_0$ and $r>0$, the preimage $\pi_{\ovl{X}}^{-1}(B(q,r))$ can be covered by at most $C$ balls in $\ovl{X}$ of radii at most $C\cdot r$.
\end{lemma}

\begin{proof}
By Lemma~2.3 in~\cite{sM09}, there exists $C\geq1$ such that $\pi_{S_2}^{-1}(B(q,r))$ can be covered by at most $C$ balls with radii at most $C\cdot r$. If $(B(p_i',r_i))$ is a family of such balls in $S_2$, then $(B(\pi_{S_2\to\ovl{X}}(p_i'),r_i))$ is the desired family in $\ovl{X}$.
\end{proof}

\begin{lemma}\label{L:Z}
The metric space $X$ is homeomorphic to the planar domain $Q$ and satisfies conditions \eqref{2 reg condition} and \eqref{compact condition}-\eqref{rel sep condition} of Theorem \ref{finite rank}. Moreover, the rank of $\partial{X}$ is $1$.
\end{lemma}

\begin{proof}
The first assertion of the lemma follows from the fact that $Q$ is locally geodesic. We leave it to the reader to verify that each non-trivial component of $\partial{X}$ is a scaled copy of  $\mathcal{S}^1$, that the collection of components of $\partial{X}$ is uniformly relatively separated, and that $\mathcal{C}(X)$ has a single limit point, which implies that the rank of $\partial{X}$ is $1$.  The completion $\ovl{X}$ is compact since $S_2$ is compact and $\pi_{S_2\to\ovl{X}}$ is continuous.

We now establish the Alhfors $2$-regularity of $X$. Since the boundary of $X$ has zero 2-measure, it is enough to show Ahlfors 2-regularity of the completion $\ovl{X}$. Let $B_{\ovl{X}}(p,r)$ be any ball with $0<r<2{\rm diam}(\ovl{X})$. Then, since $\pi_{\ovl{X}}$ is 1-Lipschitz,   
$$
\mH^2(B_{\ovl{X}}(p,r))\geq\mH^2(\pi_{\ovl{X}}(B_{\ovl{X}}(p,r))),
$$
and by the first inclusion in Lemma~\ref{L:Incl}, the right-hand side is at least $r^2/C$ for some $C\geq 1$. 

For the other inequality, we observe that by Lemma~\ref{L:Reg}, every cover of \linebreak $\pi_{\ovl{X}}(B_{\ovl{X}}(p,r))$ by open balls $\tilde B_i$ of radii $\tilde r_i$ at most some $\delta>0$ induces a cover of $B_{\ovl{X}}(p,r)$ by balls $B_j$ of radii $r_j$ at most $C\cdot \delta$ with
$$
\sum_jr_j^2\leq C^3\sum_i\tilde r_i^2.
$$
Since $\pi_{\ovl{X}}(B_{\ovl{X}}(p,r))$ is contained in the Euclidean ball of radius $r$, the Ahlfors regularity now follows.

We will check the $\rm{ALLC}$ condition in several steps.
We first check the $\rm{LLC}_1$ condition. Let $B_{X}(p,r)$ be an arbitrary ball and let $x,y\in B_X(p,r)$. Since $X$ is endowed with the path metric induced from the plane and $d(x,y)<2r$, there is a curve $\gamma$ in $X$ that connects $x$ and $y$ and such that its length is at most $2r$. Thus $E=\gamma$ is the desired continuum contained in $B(p,3r)$, i.e., $X$ satisfies the 3-LLC$_1$ condition.

Now we show that $\ovl{X}$ satisfies the LLC$_2$ condition. Let $B(p,r)$ be any ball in ${\ovl{X}}$ and let $x,y\in {\ovl{X}}\setminus B(p,r)$. We may assume that $r\leq1$.  
Let $v_x$ denote a continuum in ${\ovl X}$ that contains $x$ and projects by $\pi_{\ovl{X}}$ one-to-one onto a closed vertical  interval $I_x$ that satisfies the following properties. The end points of $I_x$ are $\pi_{\ovl{X}}(x)$ and $\pi_{\ovl{X}}(x')$ for some $x'\in v_x$ with $\pi_{\ovl{X}}(x')$ contained in the boundary of $Q_0$, so that the length of $I_x$ is not larger than the Euclidean distance from $\pi_{\ovl{X}}(p)$ to the horizontal side of the boundary of $Q_0$ that contains $\pi_{\ovl{X}}(x')$. We define $v_y$ and $y'$ similarly.  It follows from the choice of $v_x$ and $v_y$ that the distance from $p$ to $v_x$, respectively $v_y$, is at least $r/2$. Indeed, suppose the distance from $p$ to, say, $v_x$ is less than $r/2$. Then there exists $x''\in v_x$ so that $d_{\ovl{X}}(p, x'')<r/2$. Since $d_{\ovl{X}}$ is the induced metric on the completion of $(Q,d_Q)$ and $d_Q$ is the path metric induced from the plane, $d_{\ovl{X}}(x,x'')$ equals the Euclidean distance between $\pi_{\ovl{X}}(x)$ and $\pi_{\ovl{X}}(x'')$, and $d_{\ovl{X}}(p, x'')$ is at least the Euclidean distance between $\pi_{\ovl{X}}(p)$ and $\pi_{\ovl{X}}(x'')$. From the choice of $v_x$ we conclude that $d_{\ovl{X}}(x,x'')\leq d_{\ovl{X}}(p, x'')$. The triangle inequality yields a contradiction.

The points $x'$ and $y'$ are contained in some closed horizontal intervals $h_x$ and $h_y$ respectively, on the \emph{outer boundary} of $X$ (i.e., the metric boundary component of $(Q,d_Q)$ that corresponds to the boundary of $Q_0$) whose lengths are $r/4$. The distances from $h_x$ and $h_y$ to $p$ are then at least $r/4$. If $v_x\cup v_y\cup h_x\cup h_y$ is connected, we are done. Otherwise, let $l$ and $l'$ denote the two complementary components of $h_x\cup h_y$ in the outer boundary of $X$. We claim that at least one of $l$ or $l'$ is at a distance at least $r/8$ from $p$. Indeed, the sum of the distances from $p$ to $l$ and $l'$ must be at least $r/4$ because the length of every curve in $\ovl Q_0$ separating $\pi_{\ovl{X}}(v_x\cup h_x)$ from $\pi_{\ovl{X}}(v_y\cup h_y)$ must be at least $r/4$. Thus either $v_x\cup v_y\cup h_x\cup h_y\cup l$ or $v_x\cup v_y\cup h_x\cup h_y\cup l'$ is the desired continuum $E$ in $\ovl{X}$ in the LLC$_2$ condition with $\lambda=8$. 

The next step is to establish the ALLC property for  $\ovl{X}$. Let $A_{\ovl{X}}(p,r)$ denote $\ovl B_{\ovl{X}}(p,2r)\setminus B_{\ovl{X}}(p,r)$ and let $x,y\in A(p,r)$. We may assume that $r\leq 1$. If $1/2\leq r\leq 1$, then the continuum $E$ found in the proof of the LLC$_2$ condition works to conclude ALLC in this case, because the diameter of $E$ is at most 2 and thus it is at most $4r$. If  $0<r<1/2$, the proof of the existence of a desired continuum follows the lines similar to those in the proof of LLC$_2$,  but we first need to localize that argument. Indeed, first we can find a unique $n\in\N$ such that $1/2\leq 2^n r<1$. Without loss of generality we my assume that $n$ is at least three. We consider the dyadic subdivision $\mathcal D$ of $Q_0$ into squares of side length $4/2^n$. 
Let $s$ be the interior of a square in this subdivision. We denote by $s_{\ovl{X}}$ the preimage of $s$ under $\pi_{\ovl{X}}$.  Let also $\dee s_{\ovl{X}}$ denote the metric boundary of $s_{\ovl{X}}$, i.e., the closure of $s_{\ovl{X}}$ in $\ovl{X}$ less $s_{\ovl{X}}$. Note that from the definitions of $\ovl{X}$ and $\pi_{\ovl{X}}$ it follows that $\dee s_{\ovl{X}}$ is the union of four closed arcs, each isometric via $\pi_{\ovl{X}}$ to a side of the boundary of $s$. 

From the choice of $n$ and the fact that $\pi_{\ovl{X}}$ is 1-Lipschitz, it follows that the projection $\pi_{\ovl{X}}(\ovl B_{\ovl{X}}(p,2r))$ can be covered by four squares from $\mathcal D$. Moreover, they can be chosen to be the first generation dyadic subsquares of a single square of side length $8/2^n$, not necessarily dyadic.  
Let $\mathcal F$ be the family of the interiors of these four squares  
and let $K$ denote the closure in $\ovl{X}$ of
$$
\cup_{s\in\mathcal F}(s_{\ovl{X}}).
$$
The set $K$ is compact and it contains $\ovl B_{\ovl{X}}(p,2r)$.
The \emph{contour} of $K$, denoted $c(K)$, is the closure in $\ovl{X}$ of the set of all points $q$ such that $q$ belongs to $\dee s_{\ovl{X}}$ for  a unique $s\in \mathcal F$. The contour $c(K)$ is thus a union of closed arcs each of which is isometric to a horizontal or vertical side of the boundary of $s$ for some $s\in\mathcal F$. Since elements of $\mathcal F$ are interiors of the first generation dyadic subsquares of a single square, it is easy to see that $c(K)$ is connected, and thus it is a continuum. Considering various combinatorial possibilities for $c(K)$ one can easily conclude  that $c(K)$ does not have global cut points. The rest of the proof of the ALLC property for $\ovl{X}$ follows essentially the same lines as the proof of the LLC$_2$ condition where the boundary of $Q_0$ should be replaced by $c(K)$. The diameter of the resulting continuum $E$ is at most $16/2^n\leq32r$. 

Finally, given a point $p \in X$, a radius $r>0$, and points $x,y \in A_X(p,r,2r)$, we modify the above continuum $E$ to obtain one in $X$ as follows. If $E$ does not pass through the point $p_0$ in $\ovl{X}$ that projects to $(0,0)$, then it has a neighborhood that intersects only finitely many boundary components of $X$, and thus $E$ can be modified slightly so that these boundary components are avoided. If $E$ does pass though $p_0$, then $p\neq p_0$, and we can first modify $E$ in an arbitrarily small neighborhood of $p_0$ to avoid this point and then apply the above. \end{proof}

If $(X,d)$ is a metric space and $\lambda>0$, we denote by $\lambda X$ the metric space $(X, \lambda\cdot d_X)$. The following lemma and its corollary show that the space $\ovl{X}$ has a weak tangent space that contains $S_2$, see~\cite[Chapters~7,~8]{Burago}, \cite[Chapter~9]{Fractured} for the background on Gromov--Hausdorff convergence and weak tangent spaces.
\begin{lemma}\label{L:GH}
The slit carpet $S_2$ is the Gromov--Hausdorff limit of the sequence 
$
(\ovl{X}_n=2^n\pi_{\ovl{X}}^{-1}([0,1/2^n]\times[0,1/2^n])).
$
\end{lemma}

\begin{proof}
We use Theorem~7.4.12 from~\cite{Burago}. Let $\epsilon>0$ and $N\in\N$ be chosen so that $1/2^{N+1}<\epsilon$. The 1-skeleton graph of the dyadic subdivision $\tilde D_N$ of $\ovl Q_0$ into $2^{N+1}\times2^{N+1}$ subsquares \emph{pulls back} to a graph $D_N$ in $S_2$ via $\pi_{S_2}$ and graphs $D_{N,n}$ in each $\ovl{X}_n$ via $\pi_{\ovl{X}}\circ 2^{-n}$. This means that $D_N$ is a graph embedded in $S_2$ and $D_{N,n}$ is a graph embedded in $\ovl{X}_n$ such that the sets of vertices are the sets of preimages of the vertices of $\tilde D_N$ by $\pi_{S_2}$ and $\pi_{\ovl{X}}\circ 2^{-n}$ respectively. Two vertices are connected by an edge in $D_N$ or $D_{N,n}$ if and only if they are connected by an edge in $\tilde D_N$. Note that there are pairs of distinct vertices of $D_N$, respectively $D_{N,n}$, that get mapped to the same vertex of $\tilde D_N$. We do not connected such pairs of vertices by an edge. If $n\geq N$, the graphs $D_N$ and $D_{N,n}$ are identical. Since the vertices of these graphs form $\epsilon$-nets, the lemma follows.
\end{proof}

The following corollary is immediate. 
\begin{corollary}\label{C:WT}
The completion of $X$ in the path metric has a weak tangent space that contains $S_2$.
\end{corollary}

A metric Sierpi\'nski carpet $S$ is called \emph{porous} if there exists $C\geq1$ such that for every $p\in S$ and $0<r\leq{\rm diam}(S)$, there exists a peripheral circle $J$ in $S$ with $J\cap B(p,r)\neq\emptyset$ and
$$
\frac{r}{C}\leq {\rm diam}(J)\leq C\cdot r.
$$ 

\begin{lemma}\label{L:S2}
The slit carpet $S_2$ cannot be quasisymmetrically embedded into the standard plane $\R^2$.
\end{lemma}

\begin{proof}
The metric Sierpi\'nski carpet $S_2$ is porous, see~\cite[Proposition~2.4]{sM09}, and its peripheral circles are uniform quasicircles, in fact they are isometric to circles. Assume that there is a quasisymmetric embedding $f\: S_2\into\R^2$. An easy application of Proposition~10.8 in~\cite{LAMS} implies that the image $f(S_2)$ is a porous metric Sierpi\'nski carpet in the restriction of the Euclidean metric. The peripheral circles of $S_2$, and hence the boundaries of the complementary components of $f(S_2)$, are uniform quasicircles. Applying Theorem \ref{porous quasicircles} to each complementary component of $f(S_2)$ now shows that $f(S_2)$ is porous as a subset of $\R^2$. Theorem \ref{porous dim} now states that the Assouad dimension, and hence the Hausdorff dimension, of $f(S_2)$ is strictly less than two. On the other hand, according \cite[Proposition 2.4]{sM09}, $S_2$ is Ahlfors 2-regular, and it contains a curve family of positive 2-modulus~\cite[Lemma~4.2]{sM09}. This contradicts \cite[Theorem 15.10]{LAMS}.
\end{proof}

\begin{proof}[Proof of Theorem~\ref{example}]
By Lemma~\ref{L:Z}, it suffices to show that $X$ cannot be quasisymmetrically embedded into the standard 2-sphere $\Sph^2$. Suppose such a quasisymmetric embedding $f\: X\to \Sph^2$ can be found. It extends to a quasisymmetric embedding $\ovl f$ of the completion $\ovl{X}$ into $\Sph^2$ \cite[Proposition 10.10]{LAMS}. The map $\ovl f$ then induces a quasisymmetric embedding of every weak tangent space of $\ovl{X}$ into the standard plane. Corollary~\ref{C:WT} provides a weak tangent space of $\ovl{X}$ that contains $S_2$. This contradicts Lemma~\ref{L:S2}.
\end{proof}

\section{Open questions}\label{problems}
\begin{question} In Theorem \ref{finite rank}, can the assumption that $\partial{X}$ have finite rank be removed? It seems likely that this is the case. An affirmative answer is implied by an affirmative answer to the following question. Let $(X,d)$ be a metric space, homeomorphic to a domain in $\sphere$, that satisfies conditions \eqref{2 reg condition}-\eqref{rel sep condition} of Theorem \ref{finite rank}, and  has no isolated trivial boundary components. Consider the glued space $(Z,d_Z)$ formed from $\ovl{X}$ and a collection of quasidisks $\{D_i\}$ corresponding to (all) the components of $\partial{X}$, as in the proof of Theorem \ref{finite rank}. Is it true that $(Z,d_Z)$ is homeomorphic to $\sphere$? \end{question}

\begin{question} Suppose that $(X,d)$ is a metric space, homeomorphic to a domain in $\sphere$, that satisfies the $\rm{ALLC}$ condition. Are the components of $\partial{X}$ uniformly relatively separated, quanitatively? Using the techniques of Section \ref{TotDisCon}, it can be shown that the answer is ``yes" in the case that $(X,d)$ is a domain in $\sphere$. We suspect that the answer is ``no" in general. \end{question}

\begin{question} Is there a quantitative statement, analogous to Theorem \ref{finite rank}, that uniformizes onto the class of all circle domains? Given a particular circle domain $\Omega$ that does not have uniformly relatively separated boundary components, can one give sufficient intrinsic conditions for a metric space to be quasisymmetrically equivalent to $\Omega$?
\end{question}

\bibliographystyle{plain}
\bibliography{RealKoebe}
\end{document}